\documentclass[11pt, notitlepage]{article}
\usepackage{amssymb,amsmath,comment}
\catcode`\@=11 \@addtoreset{equation}{section}
\def\thesection{\arabic{section}}

\def\theequation{\thesection.\arabic{equation}}
\catcode`\@=12
\usepackage{enumitem}
\usepackage{colortbl}%
\usepackage{a4wide}
\usepackage{parskip}
\usepackage{hyperref}

\newcommand{\ds} {\displaystyle}
\newcommand{\e}{\epsilon}

\newcommand{\al} {\alpha}

\newcommand{\de} {\delta}

\newcommand{\Om} {\Omega}
\newcommand{\ra} {\rightarrow}

\newcommand{\De} {\Delta}
\newcommand{\la} {\lambda}
\newcommand{\La} {\Lambda}
\newcommand{\noi} {\noindent}
\newcommand{\na} {\nabla}

\newcommand{\mb} {\mathbb}
\newcommand{\mc} {\mathcal}

\newcommand{\I}{\int\limits_}

\usepackage[all]{xy}
\catcode`\@=11
\def\theequation{\@arabic{\c@section}.\@arabic{\c@equation}}
\catcode`\@=12

\def\QED{\hfill {$\square$}\goodbreak \medskip}
\newtheorem{Theorem}{Theorem}[section]
\newtheorem{Lemma}[Theorem]{Lemma}
\newtheorem{Proposition}[Theorem]{Proposition}

\newtheorem{Remark}[Theorem]{Remark}
\newtheorem{Definition}[Theorem]{Definition}

\def\XXint#1#2#3{{\setbox0=\hbox{$#1{#2#3}{\int}$ }
		\vcenter{\hbox{$#2#3$ }}\kern-.6\wd0}}

\begin{document}
	{\vspace{0.01in}
		\title
		{An elliptic problem involving critical Choquard and singular discontinuous nonlinearity}
		\author{ {\bf G.C. Anthal\footnote{	Department of Mathematics, Indian Institute of Technology, Delhi, Hauz Khas, New Delhi-110016, India. e-mail: Gurdevanthal92@gmail.com}\;, J. Giacomoni\footnote{ 	LMAP (UMR E2S UPPA CNRS 5142) Bat. IPRA, Avenue de l'Universit\'{e}, 64013 Pau, France.	  
					email: jacques.giacomoni@univ-pau.fr}  \;and  K. Sreenadh\footnote{Department of Mathematics, Indian Institute of Technology, Delhi, Hauz Khas, New Delhi-110016, India.	e-mail: sreenadh@maths.iitd.ac.in}}}
		\date{}
		
		\maketitle
		     
		\begin{abstract}
		\noi	The present article investigates the existence, multiplicity and regularity of weak solutions of  problems involving a combination of critical Hartree type nonlinearity along with singular and discontinuous nonlinearity (see $(\mc P_\la)$ below). By applying variational methods and using the notion of generalized gradients for Lipschitz continuous functional, we obtain the existence and the multiplicity of weak solutions for some suitable range of $\la$ and $\gamma$. Finally by studying the $L^\infty$-estimates and boundary behaviour of weak solutions, we prove their H\"{o}lder and Sobolev regularity. \\
			
			\noi \textbf{Key words:} Critical Choquard nonlinearity, Hardy-Littlewood-Sobolev inequality, existence results, discontinuous nonlinearites.\\
			
			\noi \textit{2020 Mathematics Subject Classification:} 35J20, 35J60, 35J75.
			
		\end{abstract}	
		\section{Introduction}\label{I}
		In this article, we consider the following elliptic problem involving both critical Choquard and discontinuous and singular nonlinearities. Presicely, we deal with
		\begin{align*}
			(\mc P_\la)		\begin{cases}	-\De u = \la \left( \left( \ds\I\Om \frac{|u|^{2_\mu^\ast}(y)}{|x-y|^\mu}dy\right)|u|^{2_\mu^\ast-2}u+\chi_{\{u<a\}}u^{-\gamma}\right),\\
				u> 0~\text{in}~\Om,~u \equiv 0~\text{on}~\partial \Om,
			\end{cases}
		\end{align*}
		where $\Om$ is a bounded domain of $\mb R^n$ with smooth boundary $\partial \Om$, $\gamma >0$, $n \geq 3$, $a>0$, $\la >0 $, $2_\mu^\ast =(2n-\mu)/(n-2)$, $0<\mu<n$ and $\chi_A$ denotes the characteristic function of a set $A$. \\
	An important obstacle in investigating this class of problems is that the corresponding energy functional is nondifferentiable due to the discontinuous nonlinearity. Therefore, we utilise the idea of generalized gradients as explained in the important work of F. H. Clarke (\cite{Cl}), which was later applied to the setting of partial differential equations by Chang \cite{Ch}.\\
	The occurence of discontinuous nonlinearities arises in the modelling of a number of physical issues, including the obstacle problem, the seepage surface problem, and the Elenbass equation, for further information, see \cite{Ch1, Ch2}. These significant applications have driven a long series of investigations on problems involving such nonlinearities. We mention the pioneering work of Badiale and Tarentello \cite{BT}, where existence and multiplicity results are established in the situation of critical growth and discontinuous nonlinearities in $\mb R^n$ with $n \geq 3$. We also quote further papers that consider different varieties of diffusion operators and nonlinearities, see \cite{Alb, DPST, DPST1, HaS}.\\
		The problems involving Choquard type nonlinearity are widely studied since these problems found their applications in various physical phenomena. First, S. I. Pekar \cite{Pe} used such kind of nonlinearities to describe the quantum mechanics of a polaron at rest  whereas P. Choquard \cite{Li} described the model of an electron trapped in its own hole using such nonlinearity. One of the initial study of problems involving Choquard nonlinearities using variational methods was conducted by E. H. Lieb \cite{Li} wherein he established the existence and uniqueness of a positive radial ground state of the following problem
		\begin{align*}
			-\De u + V_0 u =(I_2 \ast |u|^2)u~\text{in}~\mb R^3,
		\end{align*}
	where $I_\mu(x) =\ds \frac{A_\mu}{|x|^{n-\mu}}$ with $A_\mu =\frac{\Gamma\left(\frac{n-\mu}{2}\right)}{2^\mu \Gamma \left(\frac{\mu}{2}\right)\pi^\frac{n}{2}}$ . Without any attempt to provide the complete list, we refer to \cite{GY, MS4, MS1} and the references therein for the study of Choquard problems using variational methods.\\
	The problems involving singular nonlinearities have a very long history. These type of problems has numerous applications in the physical world, such as in the study of non-Newtonian flows in porous media and heterogeneous catalyst.
	One of the seminal breakthrough in the study of such problems was the work  \cite[Crandall-Rabinowitz-Tartar]{CRT}. By applying the method of sub-supersolutions to the non singular approximated problem and then passing to the limit, the authors proved the existence of a solution to a class of elliptic PDEs involving a singular nonlinearity. Following this pionneering work, a significant amount of study has been conducted on elliptic singular equations about existence and qualitative properties, of solutions. In this regards, we refer the survey articles \cite{GR, HM} and references therein.\\
	The investigation of singular problems in combination with critical growth nonlinearities was pioneered by  \cite[Haitao]{H} wherein the author considered the following problem:
		\begin{align}\label{ecsp}
			-\De u =\la u^{-\gamma} + u^p,~u >0~\text{in}~\Om,~u =0~\text{on}~\partial\Om
		\end{align}
		where $\Om \subset \mb R^n~(n \geq 3)$ is a smooth bounded domain and $\gamma \in (0,1)$, $1<p\leq \frac{n+2}{n-2}$. Using monotone iterations and the mountain pass lemma, the author proved existence and multiplicity results for the maximal range of parameter $\la$ (i.e. global multiplicity). Later in \cite{AG, DPST} the authors studied such problems for the higher singular cases, i.e. with $\gamma \in (1,3)$. Finally, Hirano, Saccon and Shioji in \cite{HSS} handled problem \eqref{ecsp} for any $\gamma >0$ and showed the  existence of $L^1_{\text{loc}}$ solutions $u$ satisfying $(u-\e)^+ \in H^1_0(\Om)$ for all $\e >0$ using variational methods and nonsmooth analysis arguments. We also mention the work \cite{GGS1} where the authors studied a doubly nonlocal critical singular problem in the spirit of \cite{HSS} and obtained the existence, multiplicity and regularity results.\\
		The stirring motivation to consider the problem $(\mc P_\la)$ are the works \cite{DPST} and \cite{DPST1}, where the authors discussed the problem involving critical nolinearities with singular and discontinuous nonlinearities for $n =2$ and $n \geq 3$ respectively. More precisely, in \cite{DPST1} the authors considered the following problem
		\begin{align*}
			\left\{-\De u = \la(\chi_{\{u<a\}}u^{-\gamma}+u^{2^\ast-1}),~u>0~\text{in}~\Om,~u=0~\text{on}~\partial \Om,\right.
		\end{align*}
		for $0 <\gamma<3$, $\la >0$, where $2^\ast =2n/(n-2)$ is the critical Sobolev constant and obtain the existence and multiplicity results for suitable range of $\la$.\\
		Following the above discussion, we considered the problem $(\mc P_\la)$ in the present work. The novelty features of this  work is double  with the presence of a nonlocal and critical Hartree type nonlinearity and  a more singular nonlinearity (by considering higher values of $\gamma>1$ with respect to former contributions). This brings additional technical difficulties and forces to follow a new approximation approach. Based on the notion of weak solutions given in the next section, we prove the following existence and multiplicity result:
		\begin{Theorem}\label{tmr}
			For any $a>0$, there exists $\Lambda^a >0$ such that
			\begin{enumerate}
				\item $(\mc P_\la)$ has no solution for any $\la >\Lambda^a$.
				\item $(\mc P_\la)$
				admits at least one (minimal) solution $v_\la$ for any $\la \in (0,\Lambda^a)$ and $\gamma >0$. 	Moreover for any $\omega \in (\max\{\frac{\gamma+1}{4}, 1\},\infty)$,
				$v_\la^{\omega}  \in H_0^1(\Om)$. In addtion,  $v_\la\in H^1_0(\Omega)$ if and only if $\gamma<3$.
				\item Further if we take $0 <\gamma<3$ and $\mu <\min\{4,n\}$, then $(\mc P_\la)$ admits at least two solutions for any $\la \in (0, \Lambda^a)$.
			\end{enumerate}
		\end{Theorem}
		We also discuss the boundary behavior and H\"{o}lder regularity results of weak solutions. We have the following result in this direction.
		\begin{Theorem}\label{trr}
			Let $u$ be a weak solution of $(\mc P_\la)$ and $\phi_\gamma$ be given by Definition \ref{d2.3}. Then $u \in L^\infty(\Om)$ and $C_1 \phi_\gamma\leq u\leq C_2 \phi_\gamma$ for some positive constants $C_1$ and $C_2$. Moreover, the following assertions hold:
			\begin{enumerate}[label=(\roman*)]
			\item When $\gamma >1$, $u(x) \in C^\frac{2}{\gamma+1}(\overline{\Om}).$
				\item When $\gamma =1$, $u(x) \in C^\beta(\overline{\Om})$, for all $0<\beta<1$.
				\item When $\gamma <1$, $u(x) \in C^{1,1-\gamma}(\overline{\Om})$.
			\end{enumerate}
		\end{Theorem}
		 The problems involving discontinuous nonlinearities but without singular term are tackled using variational techniques and the generalized gradient theory for locally Lipschitz functional. However the presence of singular term makes the associated energy functional neither differentiable nor locally Lipschitz in $H_0^1(\Om)$ which prohibits the use of both techniques directly. In order to overcome these difficulties we first considered the regularized problem $(\mc P_{\la,\e})$ (see Section \ref{s4}). The use of this regularization makes the associated energy functional differentiable and thus allow the use of suitable variational methods.\\
		We begin our analysis by studying the purely singular discontinuous problem $(\mc S_\la)$ (see Section \ref{s3}). To this attempt, we consider the regularized problem $(\mc S_{\la,\e})$. The analysis of $(\mc S_{\la,\e})$ is divided into two cases depending up on the parameter $\gamma$ i.e., when $(a)~0<\gamma<3$ and $(b)~\gamma \geq 3$. For Case $(a)$, we applied Perron's method and show the existence of unique solution to $(\mc S_{\la,\e})$ in $H_0^1(\Om)\cap L^\infty(\Om)$. Concerning Case $(b)$, we make use of monotone methods to obtain the existence of unique solution of $(\mc S_{\la,\e})$. The existence of the minimal weak solution of $(\mc S_\la)$ is then obtained as the limit of solutions of the regularized problem. After studying the purely singular problem, we then show the existence of a weak solution of $(\mc P_\la)$ for suitable range of $\la$ taking advantage of the construction of suitable sub and supersolutions. Under the restriction $\mu <\min\{4,n\}$, this solution is then shown to be the local minimum of the energy functional in $H_0^1(\Om)$ topology. Then the existence of second solution is obtained by investigating the translated problem associated to $(\mc P_\la)$. The associated energy functional is locally Lipschitz which leads to the use of generalized gradients technique. We further employ  Ekeland variational principle and the concentration-compactness principle to get the existence of a second solution. We pointout here that the nonsmooth analysis arguments as performed in \cite{HSS} cannot be used here because of the discontinuous term.\\
		Turning to the structure of the paper, in Section \ref{P} we collect the preliminaries required in the subsequent sections. In Section \ref{s3} we study the purely singular discontinuous problem. In Section \ref{s4}, we obtain the existence of first solution. In Section \ref{s5}, the existence of the second solution in discussed that achieves the proof of Theorem \ref{tmr}. Finally, in Section \ref{s6} we discuss the regularity of the solutions. and prove Theorem \ref{trr}.\\

		\textbf{Notations:} Throughout the paper, we will use the following notations:
		\begin{itemize}
			\item $\de(x):=\text{dist} (x,\partial \Om)$ and $d_\Om =\text{diam}(\Om)$;
			\item We denote positive constants by $M, M_1, M_2, \cdots$;
			\item We denote the standard norm on $L^p(\mb R^n)$ by $|\cdot|_p$;
			\item  for any two functions $g,~h$, we write $g \prec h$ or $g \succ h$ if there exists a constant $C>0$ such that $g \leq Ch$ or $g \geq Ch$. We write $g \sim h$ if $g \prec h$ and $ g \succ h$.
			
		\end{itemize}
		\section{Preliminaries}\label{P}
		In this section we give the functional settings and collect the notations and preliminary results required in the rest of the paper. We first define the notion of a weak solution as follows:
		\begin{Definition}
			We say that $u \in H^1_{\text{loc}}(\Om)$ is a weak solution of $(\mc P_\la)$ if
			\begin{enumerate}
				\item  ess$\inf_K u>0$ for any compact set $K \subset\Om$ 
				\item $(u-\nu)^+ \in H_0^1(\Om)$ for every $\nu >0$.
				\item For any $\psi \in C_c^\infty(\Om)$ it holds
			\begin{equation}\label{e2.3}
				\I\Om \na u\na \psi =\la \I{\Om}\chi_{\{u<a\}}u^{-\gamma}\psi +\la \I\Om\I\Om\frac{u^{2_\mu^\ast}(y)u^{2_\mu^\ast-1}(x)\psi(x)}{|x-y|^\mu}dxdy.
			\end{equation}
		\end{enumerate}
		\end{Definition} 
	\begin{Remark}
		We want to remark that the assumption $(u -\nu)^+ \in H_0^1(\Om)$ for every $\nu >0$ holds if there exists $\ell \geq 1$ such that $u^\ell \in H_0^1(\Om)$. 
	\end{Remark}
		The formal energy functional $J_\la^a(u)$ associated with the problem $(\mc P_\la)$ is given as
		\begin{align*}
			J_\la^a(u)= \frac12 \|u\|^2 -\la \I{\Om}H(u)-\frac{\la}{22_\mu^{\ast}} \I{\Om}\frac{|u(y)|^{2_\mu^\ast}|u(x)|^{2_\mu^\ast}}{|x-y|^\mu}dxdy,
		\end{align*}
	where we take
		\begin{align*}
			H(u)=\begin{cases}
				0~&\text{if}~u \leq 0,\\
				(1-\gamma)^{-1} u^{1-\gamma}~&\text{if}~0<u<\frac{a}{2},\\
				(1-\gamma)^{-1}(a/2)^{1-\gamma}+\I{a/2}^u \chi_{\{t<a\}}t^{-\gamma}dt~&\text{if}~u \geq a/2,
			\end{cases}
		\end{align*}
		for $\gamma>0$, $\gamma \ne 1$ and for $\gamma =1$ we replace the terms of the form $(1-\gamma)^{-1}x^{1-\gamma}$ in the above definition with the term $\log x$ i.e., 
		\begin{align*}
			H(u)=\begin{cases}
				0~&\text{if}~u \leq 0,\\
			\log u~&\text{if}~0<u<\frac{a}{2},\\
			\log(a/2)+\I{a/2}^u \chi_{\{t<a\}}t^{-1}dt~&\text{if}~u \geq a/2.
			\end{cases}
		\end{align*}
		\begin{Definition}\label{d2.3}
			For $0 <\gamma <\infty$ we define $\phi_\gamma$ as follows:
			\begin{align*}
				\phi_\gamma =\begin{cases}
					e_1~&0<\gamma<1,\\
					e_1(-\log e_1)^\frac12~&\gamma=1,\\
					e_1^\frac{2}{\gamma+1}~&1<\gamma,
				\end{cases}
			\end{align*}
		where $e_1$ is the first positive eigenfunction of $-\De$ on $H_0^1(\Om)$ with $|e_1|_\infty $ fixed as a number less than $1$.
		\end{Definition}
		\begin{Remark}\label{r2.4}
			If $0<\gamma<3$, by an application of Hardy's inequaltiy it follows that $u^{-\gamma}\psi \in L^1(\Om)$ if $\psi \in H_0^1(\Om)$ and $u \geq M\phi_\gamma$ in $\Om$, where $M>0$ is a constant. In particular, if $u \geq M \phi_\gamma$, then \eqref{e2.3} holds for all $\psi \in H_0^1(\Om)$.
		\end{Remark}
		Now we recall the Hardy-Littlewood-Sobolev inequality which is the foundation in study of problems involving Choquard nonlinearity:
	\begin{Proposition}\label{hls}
		\textbf{Hardy-Littlewood-Sobolev inequality:} Let $r, q>1$ and $0<\mu<n$ with $1/r+1/q+\mu/n=2$, $g\in L^{r}(\mb R^n), h\in L^q(\mb R^n)$. Then, there exists a sharp constant $C(r,q,n,\mu)$ independent of $g$ and $h$ such  that 
		\begin{equation*}\label{hlse}
			\int\limits_{\mb R^n}\int\limits_{\mb R^n}\frac{g(x)h(y)}{|x-y|^{\mu}}dxdy \leq C(r,q,n,\mu) |g|_r|h|_q.
		\end{equation*}                             
	\end{Proposition}
	In particular, let $g = h = |u|^p$ 
	then by Hardy-Littlewood-Sobolev inequality we see that,
	$$                               \int\limits_{\mb R^n}\int\limits_{\mb R^n}\frac{|u(x)|^pu(y)|^p}{|x-y|^{\mu}}dxdy$$
	is well defined if $|u|^p \in L^\nu(\mb R^n)$ with $\nu =\frac{2n}{2n-\mu}>1$.
	Thus, from Sobolev embedding theorems, we must have
	\begin{equation*}
		\frac{2n-\mu}{n} \leq p \leq \frac{2n-\mu}{n-2}.
	\end{equation*}
	From this, for $u \in L^{2^*}(\mb R^n) $ we have
	$$    \left(\int\limits_{\mb R^n}\int\limits_{\mb R^n}\frac{|u(x)|^{2_\mu^\ast}|u(y)|^{2_\mu^\ast}}{|x-y|^\mu}dxdy \right)^\frac{1}{2_\mu^\ast} \leq C(n,\mu)^\frac{1}{2_\mu^\ast}  |u|_{2^*}^2  .                                     $$
	We fix $S_{H,L}$ to denote the best constant associated to Hardy-Littlewood-Sobolev inequality, i.e,
	\begin{align*}\label{ebc}	S_{H,L}=\inf\limits_{u \in C_0^\infty(\mb R^n)\setminus \{0\}} \frac{\|\nabla u\|_{L^2(\mb R^n)}^2}{\|u\|_{HL}^2}.
	\end{align*}
Now the following lemma plays a crucial role in the sequel:
\begin{Lemma}\label{l2.3}
	The constant $S_{H,L}$ is achieved if and only if 
	\begin{equation*}
		u=C\left(\frac{b}{b^2+|x-d|^2}\right)^\frac{n-2}{2},
	\end{equation*}
	where $C>0$ is a fixed constant, $d  \in \mb R^n$ and $b \in (0,\infty)$ are parameters. Moreover,
	\begin{equation*}
		S= C(n,\mu)^\frac{n-2}{2n-\mu}S_{H,L}.
	\end{equation*}
\end{Lemma}

			\section{The purely singular Discontinuous problem}\label{s3}
		In order to prove the existence results for $( P_\la)$, we translate the problem by the minimal solution to the purely singular problem:
		\begin{equation*}
			( \mc S_\la)\begin{cases}\mc  -\De u =\la\chi_{\{u<a\}}u^{-\gamma},~u>0~\text{in}~\Om,~u=0~\text{on}~\partial \Om.
			\end{cases}
		\end{equation*}
		We first study the existence of weak solutions to $(\mc S_\la)$. We have the following result in this direction
		\begin{Proposition}\label{p3.1}
			There exists a weak minimal solution $u_\lambda$ of $(\mc S_\la)$ for any $\gamma >0$. Furthermore, we have $u_\la \sim \phi_\gamma$ near $\partial \Om$. Moreover regularity results as stated in Theorem \ref{trr} hold.
		\end{Proposition}
		The main idea is to solve an approximating regular problem that admits a unique solution which is a strict subsolution to $(\mc S_\la)$ and then pass through the limit. The approximating regular problem is obtained by replacing $\chi_{\{u <a\}}$ with the continuous function $\chi_\e(u-a)$ where 
		\begin{equation}\label{eq3.2}
			\chi_\e(t) =\chi_{(-\infty,-\e)}(t)-t\e^{-1}\chi_{[-\e,0)}(t),~t \in \mb R.
		\end{equation}
		So we consider the following problem
		\begin{align*}
			(\mc S_{\la,\e})	\begin{cases}	-\De v =\la\chi_\e(v-a)v^{-\gamma}~\text{in}~\Om,\\
				v \equiv 0~\text{on}~\partial\Om,~v> 0~\text{in}~\Om.
			\end{cases}
		\end{align*}
		
		We have the following existence results concerning $(\mc S_{\la,\e})$.
		\begin{Proposition}\label{p3.2} We have
			\begin{enumerate}[label=(\arabic*)]
				\item Let $ 0<\gamma <3$. Then there exists an $\e_0=\e_0(a)$ such that for $0<\e <\e_0$ there exists a unique solution $ u_{\la,\e} \in H_0^1(\Om) \cap L^\infty(\Om)$ of the problem $(\mc S_{\la,\e})$. Also the map $\e \ra u_{\la,\e}$ is nonincreasing. Moreover, we have $u_{\la,\e} \sim \phi_\gamma$ for any $\e \in (0,\e_0)$.
				\item Let $\gamma \geq 3$. Then there exists unique solution $u_{\la,\e} \in H^1_{\text{loc}}(\Om )$ of $(\mc S_{\la,\e}) $ such that for any compact $K \Subset \Om$, there exists $M(K)>0$ which satisfies $u_{\la,\e} \geq M(K)>0$ a.e. in $K$. Moreover 
					\begin{align}\label{3.3}
						u_{\la,\e}~\text{is uniformly bounded in}~L^\infty(\Om),
					\end{align}
					and 
				\begin{align}\label{equ3.2}
					u_{\la,\e}^{\ell}~\text{is uniformly bounded in}~H_0^1(\Om)~\text{with}~\ell> \frac{\gamma+1}{4}.
				\end{align}
			\end{enumerate}
			Also the map $\e \ra u_{\la,\e}$ is nonincreasing and $u_{\la,\e} \sim \phi_\gamma$ independent of $\e$.
		\end{Proposition}
		We will consider the above two cases separately. First we consider the case when $0<\gamma<3$. In this case, the proof
		is based on the classical sub-supersolution method in a variational setting called Perron's method. In this direction,
		inspired by Haitao \cite{H}, we establish the following result.
		\begin{Lemma}\label{l3.3}
			Let $0<\epsilon<\frac{a}{2}$. Suppose that $\underline{u},~\overline{u} \in H_0^1(\Om) \cap L^\infty(\Om)$ are weak subsolution and supersolution of $(\mc S_{\la,\e})$ respectively, such that $0 <\underline{u} \leq \overline{u}$ in $\Om$ and $\underline{u} \geq M(K)>0$ for every $K \Subset \Om$, for some constant $M(K)$. Then there exists a solution $u \in H_0^1(\Om) \cap L^\infty(\Om)$ of $(\mc S_{\la,\e})$ satisfying $\underline{u} \leq u\leq \overline{u}$ in $\Om$.
		\end{Lemma}
		\begin{proof}
			For sake of clarity, we give only the proof for $\gamma\neq 1$. We introduce the energy functional associated to $(\mc S_{\la,\e})$ :
			\begin{equation*}
				S_\la^{a,\e}(u)=\frac12 \|u\|^2 -\la\I{\Om}H_{\e}(u)dx,
			\end{equation*}
			where 
			\begin{align}\label{3.4}
				H_\e(t)=\begin{cases}
					0~&\text{if}~u \leq 0,\\
					(1-\gamma)^{-1} t^{1-\gamma}~&\text{if}~0<t<\frac{a}{2},\\
					(1-\gamma)^{-1}(a/2)^{1-\gamma}+\I{a/2}^t \chi_\e(s-a)s^{-\gamma}dt~&\text{if}~t \geq a/2.
				\end{cases}
			\end{align}
			Observe that the map $\chi_\e(t)$ lies in $[0,1]$, is continuous and nonincreasing. Also we see that the map $\chi_\e(t)t^{-\gamma}$ is nonincreasing. Concerning $H_\e(t)$, we observe that $ H_\e(t) \leq (1-\gamma)^{-1  } t^{1-\gamma}$ for $t>0$. Now let us consider the conical shell set defined as:
			\begin{align*}
				\mc M =\{v \in H_0^1(\Om) : \underline{u} \leq v \leq \overline{u}~\text{in}~\Om\}.
			\end{align*}
			Clearly, $ \mc M \ne \emptyset$ is closed and convex. Therefore $S_\la^{a,\e}$ is weakly sequentially lower semicontinuous over $\mc M$. Indeed, it is enough to show that $S_\la^{a,\e}$ is sequentially lower semicontinuous. Let $\{v_k\}_{k \in \mb N} \subset \mc M$ be such that $v_k \to v$ in $H_0^1(\Om)$. Now observe that
			\begin{align*}
				|H_\e(v_k)|  \leq \begin{cases}
					\left((1-\gamma)^{-1}(a/2)^{1-\gamma}\right)^- + |(1-\gamma)^{-1}|\underline{u}^{1-\gamma}~&\text{if}~\gamma> 1,\\
					\left((1-\gamma)^{-1}(a/2)^{1-\gamma}\right)^- +|(1-\gamma)^{-1}|\overline{u}^{1-\gamma}~&\text{if}~0<\gamma<1.
				\end{cases}
			\end{align*}
			Using the fact that $ 0 <\underline{u} \leq \overline{u} \in H_0^1(\Om) \cap L^\infty(\Om)$ and $\Om$ is bounded, we conclude that sequence $|H_\e(v_k)|$ is bounded by a $L^1(\Om)$ function (thanks to $\gamma<3$). Thus by using the dominated convergence theorem and the continuity of the norm, we obtain the required claim. Therefore there exists $u \in \mc M$ such that 
			\begin{equation*}
				S_\la^{\e,a}(u) =\inf_{v\in \mc M} S_\la^{\e,a}(v).
			\end{equation*}
			Next, we show that $u$ is the required weak solution of 
			$(\mc S_{\la,\e})$. For this, let $\varphi \in C_c^\infty(\Om)$ and $\kappa >0$. We define
			\begin{align*}
				\eta_{\kappa } =\begin{cases}
					\overline{u}~&\text{if}~u+\kappa \varphi \geq \overline{u},\\
					u +\kappa \varphi~&\text{if}~\underline{u} \leq u+\kappa  \varphi \leq \overline{u},\\
					\underline{u}~&\text{if}~u+\kappa  \varphi \leq \underline{u}.
				\end{cases}
			\end{align*}
			Observe that $ \eta_{\kappa }= u+\kappa \varphi -\varphi^{\kappa } +\varphi_{\kappa } \in \mc M$, where $ \varphi^{\kappa } =(u+\kappa  \varphi -\overline{u})^+$ and $\varphi_{\kappa }=(u +\kappa  \varphi -\underline{u})^-$. Since $u$ is a minimizer of $S_\la^{\e,a}$ over $\mc M$, we have
			\begin{align}\nonumber \label{eq3.4}
				0 \leq& \lim_{t \ra 0} \frac{S_\la^{\e,a}(u+t(\eta_{\kappa }-u))-S_\la^{\e,a}(u)}{t}\\
				=& \I\Om \na u\na (\eta_{\kappa}-u)dx -  \la\I{\Om} (\eta_{\kappa } -u)\chi_\e(u-a)u^{-\gamma}dx.
			\end{align}
			Using the definition of $\eta_{\kappa }$, $\varphi_{\kappa }$ and $\varphi^{\kappa}$, from \eqref{eq3.4}, we have
			\begin{align}\label{e3.5}
				\I\Om \na u\na \varphi dx - \la\I{\Om}\chi_\e(u-a)u^{-\gamma}\varphi dx\geq \frac{1}{\kappa}(\mc G^{\kappa} -\mc G_{\kappa }),
			\end{align}
			where 
			\begin{align*}
				\mc G^{\kappa } =\I\Om \na u\na \varphi^{\kappa} dx - \la\I{\Om}\chi_\e(u-a)u^{-\gamma}\varphi^{\kappa } dx
			\end{align*}
			and
			\begin{align*}
				\mc G_{\kappa } =\I\Om \na u\na \varphi_{\kappa } dx  - \la\I{\Om}\chi_\e(u-a)u^{-\gamma}\varphi_{\kappa } dx.
			\end{align*}
			Next we will give estimates of $\mc G^{\kappa }$ and $\mc G_{\kappa }$. First we give:\\
			\textbf{Estimate of $\mc G^{\kappa }$}: Setting $\Om^{\kappa } =\{ \varphi^{\kappa }>0\}$, we have
			\begin{align}\nonumber
				\frac{1}{\kappa} \I\Om \na u\na \varphi^{\kappa }=&\frac{1}{\kappa}\I\Om \na (u-\overline{u})\na \varphi^{\kappa} dx +\frac{1}{\kappa } \I\Om \na \overline{u}\na \varphi^{\kappa} dx\\ \nonumber
				=& \frac{1}{\kappa }\I{\Om^\kappa} \na (u-\overline{u})\na(u+\kappa\varphi -\overline{u}) dx +\frac{1}{\kappa} \I\Om \na \overline{u}\na \varphi^{\kappa} dx\\ \nonumber
			\end{align}
		\begin{align}\label{eq3.6}
				\geq & \I{\Om^{\kappa }} \na (u-\overline{u})\na \varphi dx+\frac{1}{\kappa} \I\Om \na \overline{u}\na \varphi^{\kappa}  dx =o(1) +\frac{1}{\kappa} \I\Om \na \overline{u}\na \varphi^{\kappa} dx.
			\end{align} 
			Using \eqref{eq3.6} and employing  the facts that $\overline{u}$ is a supersolution of $(\mc S_{\e})$, $\chi_\e(t)t^{\gamma}$ is nonincreasing and  $(\chi_\e(\overline{u}-a) \overline{u}^{-\gamma}-\chi_\e(u-a) u^{-\gamma})\varphi \in L^1(\Om)$, together with dominated convergence theorem we obtain  the following
			\begin{align} \nonumber\label{eq3.8}
				\frac{1}{\kappa} \mc G^{\kappa}\geq& o(1) +\frac{1}{\kappa}\left( \I\Om \na \overline{u}\na \varphi^{\kappa } dx -  \la\I{\Om}\chi_\e(u-a)u^{-\gamma}\varphi^{\kappa } dx\right)\\ \nonumber
				=& o(1) +\frac{1}{\kappa}\left( \I\Om \na \overline{u}\na \varphi^{\kappa} dx-  \la\I{\Om}\chi_\e(\overline{u}-a)\overline{u}^{-\gamma}\varphi^{\kappa} dx\right)\\ \nonumber
				&+ \frac{\la}{\kappa } \left(\I{\Om}\chi_\e(\overline{u}-a)\overline{u}^{-\gamma}\varphi^{\kappa} dx-\I{\Om}\chi_\e(u-a){u}^{-\gamma}\varphi^{\kappa} dx \right)\\ \nonumber
				\geq & o(1) +\frac{\la}{\kappa} \left(\I{\Om^{\kappa }}(\chi_\e(\overline{u}-a)\overline{u}^{-\gamma }-\chi_\e(u-a){u}^{-\gamma })(u-\overline{u}) dx\right)\\
				&+ \la\I{\Om^{\kappa }}\left(\chi_\e(\overline{u}-a){\overline{u}}^{-\gamma}- \chi_\e(u-a)u^{-\gamma}\right)\varphi dx \geq o(1).
			\end{align}
			In the similar fashion, we see that
			\begin{align}\label{eq3.9}
				\frac{1}{\kappa} \mc G_{\kappa}\leq o(1).
			\end{align}
			Using \eqref{eq3.8} and \eqref{eq3.9}, we conclude from \eqref{e3.5}
			\begin{align*}
				0 \leq \I\Om \na u \na \varphi dx - \la \I\Om \chi_\e(u-a) u^{-\gamma} \varphi dx,
			\end{align*}
			and the claim follows using the arbitrariness of $\varphi$. \QED
		\end{proof}
		\textbf{Proof of Proposition \ref{p3.2} $(1)$:} In view of Lemma \ref{l3.3}, we construct an ordered pair of sub- and supersolution $\underline{u}$ and $\overline{u}$, respectively, of $(\mc S_{\la,\e})$ for $\e$ small enough. Choose $0 <\e_0 <a$. Then for $0 <\e <\e_0$, we see that $ \chi_\e(t-a)t^{-\gamma} \ra \infty$ uniformly as $ t \ra 0$. Thus we can find $ \theta >0$ sufficiently small so that
		\begin{align*}
			\la_1 \theta |e_1|_\infty\leq \la\chi_\e (\theta |e_1|_\infty-a)(\theta|e_1|_\infty)^{-\gamma}.
		\end{align*}
		Now using the fact that $ \chi_\e(t-a)t^{-\gamma}$ is nonincreasing, we obtain by taking $\underline{u}=\theta e_1$ the following
		\begin{align*}
			-\Delta\underline{u} =\la_1(\theta e_1) \leq \la\chi_\e(\theta |e_1|_\infty-a)(\theta |e_1|_\infty)^{-\gamma} \leq\la \chi_\e(\underline{u} -a)\underline{u}^{-\gamma}.
		\end{align*}
		For supersolution, we take $\overline{u} =\hat{u}$, where $\hat{u}$ is the solution of the purely singular problem (replacing the discontinuity term by 1).
		Finally, we choose $\theta$ small enough so that $0 <\underline{u} \leq \overline{u}$ a.e in $\Om$. Applying Lemma \ref{l3.3} we obtain the existence of $u_{\la,\epsilon}$. It is easy to see that $u_{\la,\e} \sim  \phi_{\gamma}$. Uniqueness of $u_{\la,\epsilon}$, for $\gamma <3$ follows from the nonincreasing nature of $t\to \chi_\epsilon(t-a)t^{-\gamma}$ on $(0,+\infty)$. Finally, we prove that the map $\e \ra u_{\la,\e}$ is nonincreasing. For this, let $\e_1<\e_2$ and $u_{\la,\e_1},~u_{\la,\e_2}$ be the corresponding solutions. Observe that for $0<\epsilon_1<\epsilon_2$, $u_{\la,\epsilon_1}$ is a supersolution  for $(\mc S_{\la,\epsilon_2})$, from uniqueness and Lemma \ref{l3.3} we get $u_{\la,\epsilon_2}\leq u_{\la,\epsilon_1}$. \QED
		Next we consider the case $\gamma \geq 3$.  We first show the validity of  a weak comparison principle which will be used to get  the uniqueness of the solution. Here we adopt the ideas developped in \cite{CMSS}. \\
		We define the real valued function $g_k(\tau)$ by
		\begin{align*}
			g_k(\tau)= \begin{cases}
				\max\{-\la\chi_\e(\tau-a)\tau^{-\gamma},-k\}~&\text{if}~\tau>0,\\
				-k~&\text{if}~\tau \leq 0.
			\end{cases}
		\end{align*}
		Now consider the real valued function $G_k(\tau)$ defined by 
		\begin{align*}
			\begin{cases}
				G_k^\prime(\tau) =g_k(\tau),\\
				G_k(1) =0. 
			\end{cases}
		\end{align*}
		Finally, we define the functional $\Phi_k :H_0^1(\Om) \ra [-\infty,+\infty]$ by
		\begin{align*}
			\Phi_k(v)=\frac{1}{2} \|v\|^2 +\I{\Om} G_k(v)dx,~v \in H_0^1(\Om).
		\end{align*}
		Let $u$ be a fixed supersolution of $(\mc S_{\la,\e})$ and consider $w$ as the minimum of the functional $\Phi_k$ on the convex set
		\begin{align*}
			\mc M := \{v \in H_0^1(\Om): 0\leq v \leq u~\text{ a.e. in}~\Om\}.
		\end{align*}
		By \cite{KS}, it follows that
		\begin{align}\label{equa3.11}
			\I\Om \na w \na (v-w)dx \geq -\I{\Om}G_k^\prime(w)(v-w) dx~\text{for}~v \in w+(H_0^1(\Om) \cap L^\infty_c(\Om))~\text{and}~0 \leq v \leq u,
		\end{align}
		where $L^\infty_c(\Om)$ denotes the space of $L^\infty$ functions with compact support in $\Om$. 
		\begin{Lemma}\label{l3.4}
			We have that
			\begin{align}\label{equa3.12}
				\I\Om \na w\na vdx \geq -\I{\Om}G_k^\prime(w)v~\text{for}~v \in C_c^\infty(\Om)~\text{with}~v \geq 0~\text{ in}~\Om.
			\end{align}
		\end{Lemma}
		\begin{proof}
			To prove this let us consider $\psi \in C_c^\infty(\mb R )$ with $0 \leq \psi \leq 1$ for $t \in \mb R$, $\psi(t )=1$ for $t \in [-1,1]$ and $\psi(t) =0$ for $t \in (-\infty,2] \cup [2,\infty)$. Then for any $\varphi \in C_c^\infty(\Om)$ with $\varphi \geq 0$ in $\Om$, we set
			\begin{align*}
				\varphi_k := \psi \left(\frac{w}{k}\right)\varphi, ~\varphi_{k,t} =\min\{w+t\varphi_k,u\},
			\end{align*} 
			with $ k\geq 1$ and $t >0$. We have that $\varphi_{k,t} \in w+(H_0^1(\Om) \cap L_c^\infty(\Om))$ and $w \leq \varphi_{k,t} \leq u$, so that by \eqref{equa3.11} we have
			\begin{align}\label{equa3.13}
				\I\Om \na w \na (\varphi_{k,t}-w)dx \geq \I{\Om} G_k^\prime(w)(\varphi_{k,t}-w)dx. 
			\end{align}                                     
			Now using \eqref{equa3.13}, we have
			\begin{align} \nonumber
				\|\varphi_{k,t} -w \|^2 +& \I\Om \left(G_k^\prime (\varphi_{k,t})-G_k(w)\right)(\varphi_{k,t}-w)dx\\ \nonumber
				\end{align}
					\begin{align}\label{equa3.14} \nonumber
				=&\I\Om \na\varphi_{k,t} \na(\varphi_{k,t}-w)dx-\I\Om \na w \na(\varphi_{k,t}-w)dx
				+ \I\Om G_k^\prime(\varphi_{k,t})(\varphi_{k,t}-w)dx\\ \nonumber
				&-\I\Om G_k^\prime(w)(\varphi_{k,t}-w)dx\\ \nonumber
				\leq& \I\Om \na\varphi_{k,t}\na(\varphi_{k,t}-w)dx  + \I\Om G_k^\prime(\varphi_{k,t})(\varphi_{k,t}-w)dx\\\nonumber
				=&  \I\Om \na\varphi_{k,t}\na(\varphi_{k,t}-w-t\varphi_k)dx + \I\Om G_k^\prime(\varphi_{k,t})(\varphi_{k,t}-w-t\varphi_k)dx\\
				&+ t \I\Om \na\varphi_{k,t}\na\varphi_kdx +t\I\Om G_k^\prime (\varphi_{k,t})\varphi_kdx.
			\end{align}
			Now using the definition of $\varphi_{k,t}$, we easily see that
			\begin{align}\label{equa3.15}
				\I\Om \na \varphi_{k,t} \na (\varphi_{k,t}-w-\varphi_k)dx =\I\Om \na u \na (\varphi_{k,t}-w-\varphi_k)dx.
			\end{align}
			Using \eqref{equa3.15} in \eqref{equa3.14}, we get
			\begin{align}\nonumber \label{equa3.17}
					\|\varphi_{k,t} -w \|^2 +& \I\Om \left(G_k^\prime (\varphi_{k,t})-G_k(w)\right)(\varphi_{k,t}-w)dx\\ \nonumber
					\leq&  \I\Om \na u \na(\varphi_{k,t}-w-t\varphi_k)dx + \I\Om G_k^\prime(u)(\varphi_{k,t}-w-t\varphi_k)dx\\ 
				+& t\left( \I\Om \na\varphi_{k,t}\na\varphi_kdx +\I\Om G_k^\prime (\varphi_{k,t})\varphi_kdx\right).
			\end{align}
			Note now that, by the definition of $G_k$, it follows that $u$ is also a supersolution to the equation $ -\De z =-G_k^\prime (z)$, so that by observing  $\varphi_{k,t} -w-t \varphi_k \leq 0$, we deduce from \eqref{equa3.17} 
			\begin{align*}
				\|\varphi_{k,t} -w \|^2 + \I\Om \left(G_k^\prime (\varphi_{k,t})-G_k(w)\right)(\varphi_{k,t}-w)dx	
				\leq  t \I \Om \na \varphi_{k,t}\na\varphi_kdx +t\I\Om G_k^\prime (\varphi_{k,t})\varphi_kdx.
			\end{align*}
			Exploiting again the fact that $\varphi_{k,t}-w \leq t \varphi_k$, by simple computations we deduce that
			\begin{align*}
				\I\Om \na \varphi_{k,t}\na\varphi_kdx +\I\Om G_k^\prime (\varphi_{k,t})\varphi_kdx \geq - \I\Om |G_k^\prime (\varphi_{k,t})-G_k(w)||\varphi_k|dx.
			\end{align*}
			Now passing to the limit $t \ra 0$ and employing Lebesgue Dominated convergence theorem, we see that
			\begin{align*}
				\I\Om \na w\na\varphi_k dx +\I\Om G_k^\prime (w)\varphi_kdx \geq 0.
			\end{align*}
			Finally the proof holds by tending $k$ to infinity. \QED
		\end{proof}
		Next we prove a weak comparison principle from which the uniqueness of the solution follows. Precisely, we have:
		\begin{Theorem}\label{wcp}
			Let $\gamma >0$, $v$ be a subsolution to $(\mc S_{\la,\e})$ such that $(v-\nu)^+ \in H_0^1(\Om)$ for every $\nu >0$  and let $u$ be a supersolution to $(\mc S_{\la,\e})$. Then $v \leq u$ a.e. in $\Om$.
		\end{Theorem}
		\begin{proof}
			Let $w$ be  as in Lemma \ref{l3.4}. Since $w \in H_0^1(\Om)$ and nonnegative, for any $\nu >0$, supp$(v-w-\nu)^+$ is contained in supp$(v-\nu)^+$. From this, we conclude that
			\begin{equation*}
				(v-w-\nu)^+ \in H_0^1(\Om)~\text{for any}~\nu >0.  
			\end{equation*}
			Now using standard density arguments, we obtain from \eqref{equa3.12} that
			\begin{equation}\label{equa3.18}
				\I \Om  \na w \na K_\tau((v-w-\nu)^+)dx \geq \I\Om G_k^\prime(w) K_\tau((v-w-\nu)^+)dx
			\end{equation}
			where $K_\tau(t) :=\min\{t,\tau\}$ for $\tau \geq 0$ and $K_\tau(-t):=- K_\tau(t)$ for $t <0$. Let now $\psi_k \in C_c^\infty(\Om)$ such that $\psi_k \ra (v-w-\nu)^+ \in H_0^1(\Om)$ and set
			\begin{align*}
				\widetilde{\psi}_{\tau,k} := K_\tau (\min\{(v-w-\nu)^+,\psi_k^+\}).
			\end{align*}
			It follows that $\widetilde{\psi}_{\tau,k} \in H_0^1(\Om) \cap L^\infty_c(\Om)$ and by a density argument
			\begin{align*}
				\I\Om \na v \na K_\tau((v-w-\nu)^+)dx \leq \I\Om \la\frac{\chi_\e(v-a)}{v^{\gamma}} \widetilde{\psi}_{\tau,k}dx.
			\end{align*}
			Passing to the limit as $k \ra \infty$, we obtain
			\begin{equation}\label{equa3.19}
				\I\Om \na v \na K_\tau((v-w-\nu)^+)dx \leq \I\Om \la\frac{\chi_\e(v-a)}{v^{\gamma}} K_\tau((v-w-\nu)^+)dx.
			\end{equation}
		 Choosing $\nu >0$ such that $\nu^{-\gamma} <k$, from \eqref{equa3.18} and \eqref{equa3.19} we deduce that
			\begin{align*}
				\|K_\tau((v-w-\nu)^+)\|^2 \leq& \I\Om \left(\la\frac{\chi_\e(v-a)}{v^{\gamma }}+G_k^\prime(w)\right)K_\tau((v-w-\nu)^+)dx\\
				\leq & \I\Om(-G_k^\prime(v)+G_k^\prime(w)) K_\tau((v-w-\nu)^+)dx\leq 0.
			\end{align*}
			By the arbitrariness of $\tau$ we deduce that
			\begin{align*}
				v \leq w+\nu \leq u+\nu~\text{a.e. in}~\Om
			\end{align*}
			and the conclusion follows letting $\nu \ra 0$. \QED
		\end{proof} 
		Regarding the existence of solutions, we use the classical approach of regularizing the singular nonlinearities $u^{-\gamma}$ by $\left(u+\frac{1}{k}\right)^{-\gamma}$ and derive uniform a priori estimates for the weak solution of the regularized problem. More precisely, we study the following approximated problem 
		\begin{align*}
			(\mc S_{\la,\e,k})	\begin{cases} -\De u =\ds\la\frac{\chi_\e(u-a)}{(u+\frac{1}{k})^{\gamma}} , \quad u>0~&\text{in}~\Om,\\
				u=0~&\text{on}~\partial\Om. \end{cases}
		\end{align*}
		\begin{Lemma}\label{le3.4}
			For any $k \in \mb N\backslash\{0\}$ and $ \gamma>0$, there exists a unique nonnegative weak solution $u_{\la,k,\e} \in H_0^1(\Om)$ of the problem $(\mc S_{\la,\e,k})$ in the sense that
			\begin{align}\label{equ3.9}
				\I\Om \na u_{\la,k,\e} \na vdx =\la\I\Om \frac{\chi_\e(u_{\la,k,\e}-a)}{(u_{\la,k,\e}+\frac{1}{k})^{\gamma}}vdx~\text{for all}~v\in H_0^1(\Om).
			\end{align} 
			Moreover,
			\begin{enumerate}[label=\roman*)]
				\item The solution $u_{\la,k,\e} \in C^{1,\al}(\overline{\Om})$ for every $\al \in (0,1)$ and $u_{\la,k,\e} >0$ in $\Om$.
				\item The sequence $\{u_{\la,k,\e}\}_{n\in \mb N}$ is monotonically increasing in the sense that $ u_{\la,k+1,\e} \geq u_{\la,k,\e}$ for all $k \in \mb N$.
				\item For every compact set $ K \Subset \Om$ and $k \in \mb N$, there exists a constant $M(K)>0$ independent of $k$ such that $u_{\la,k,\e} \geq M(K) >0$.
				\item $u_{\la,k,\e}$ is uniformly bounded in $L^\infty(\Om)$ both in $k$ and $\e$.
				\item 
				$u_{\la,k,\e}^\omega$ is bounded in $H_0^1(\Om)$ with $\omega>\frac{\gamma+1}{4}$, independently of both $k$ and $\e$.   
			\end{enumerate}
		\end{Lemma}
		\begin{proof} Proof of parts $i)$, $ii)$, $iii)$ and $iv)$ are standard and hence omitted. We next give the proof of part $v)$. Since $u_{\la,k,\e} \in L^\infty(\Om) \cap H_0^1(\Om)$ and positive, for any $\nu >0$ and $\omega>0$, $(u_{\la,k,\e}+\nu)^\omega -\nu^\omega$ belongs to $H_0^1(\Om)$ and so by taking it as a test function in \eqref{equ3.9} with $\nu \in (0,\frac{1}{k})$ and $\omega \in [\gamma,\infty)$, we obtain
			\begin{align}\label{3.11}
				\I\Om \na u_{\la,k,\e} \na (u_{\la,k,\e}+\nu)^\omega dx 
				\leq \la\I\Om \frac{\chi_\e(u_{\la,k,\e}(x)-a)}{(u_{\la,k,\e}+\frac{1}{k})^{\gamma }} (u_{\la,k,\e}+\nu)^\omega dx \leq \la \I\Om (u_{\la,k,\e}+\nu)^{\omega-\gamma } dx.
			\end{align}
			Passing $\nu \ra 0$ in \eqref{3.11} via Fatou's Lemma, we obtain
			\begin{align}\label{3.12}
				\frac{4 \omega}{(\omega+1)^2} \I\Om |\na u_{\la,k,\e}^\frac{\omega+1}{2}|^2 dx  \leq  \la\I\Om (u_{\la,k,\e})^{\omega-\gamma} dx.
			\end{align} 
			In order to estimate the R.H.S term of \eqref{3.12}, we choose $\omega $ that satisfies (according to boundary behaviour of $u_{k,\epsilon} $):
			\begin{align*}
				\omega >\frac{\gamma-1}{2}.
			\end{align*}
		  Now using the boundary behaviour of $u_{k,\epsilon}$ 
			\begin{align}\label{3.14}
				\I\Om u_{\la,k,\e}^{\omega-\gamma} dx \leq \la\I\Om e_1^\frac{2(\omega -\gamma)}{\gamma+1}<\infty, 
			\end{align}
		since $\frac{2(\omega -\gamma)}{\gamma+1} >-1$. Combining \eqref{3.12} and \eqref{3.14}, we obtain the required conclusion.\\
			We now turn our attention to the $L^\infty$ estimates (assertion iv)). For this fix $p > \frac{n}{2}$. Now take $\phi_m(u_{\la,k,\e}):=(u_{\la,k,\e}-m)^+$ with $m \geq 1$ as a test function in \eqref{equ3.9} and using Sobolev embeddings and H\"{o}lder inequality, we get
			\begin{align} \nonumber\label{3.17}
				\left(\I{T_m} |\phi_m(u_{\la,k,\e})|^\frac{2n}{n-2}\right)^\frac{n-2}{2} \leq& M\I{T_m}|\na \phi_m(u_{\la,k,\e})|^2dx =M\I{\Om} \na u_{\la,k,\e}\na \phi_m(u_{\la,k,\e})dx\\ \nonumber
				\leq& M\I\Om \frac{\chi_\e(u_{\la,k,\e}-a)}{(u_{\la,k,\e}+\frac{1}{k})^{\gamma}}\phi_m(u_{\la,k,\e})dx \leq M \I{T_m} \chi_\e(u_{\la,k,\e}-a) \phi_m(u_{\la,k,\e})dx \\ \nonumber
				\leq& M |\chi_\e(u_{\la,k,\e}-a)|_p	\left(\I{T_m} |\phi_m(u_{\la,k,\e})|^\frac{2n}{n-2}\right)^\frac{n-2}{2n}|T_m|^{1-\frac{n-2}{2n}-\frac{1}{p}}\\
				\leq& M|\Om|^\frac{1}{p} \left(\I{T_m} |\phi_m(u_{\la,k,\e})|^\frac{2n}{n-2}\right)^\frac{n-2}{2n}|T_m|^{1-\frac{n-2}{2n}-\frac{1}{p}}
			\end{align}
			where $T_m:= \{x \in \Om: u_{\la,k,\e} \geq m\}$. Let $j >m \geq 1$, then $T_j \subset T_m$ and $\phi_m(u_{\la,k,\e}) \geq j-m$ for $x \in T_j$. Using above facts, from \eqref{3.17}, we obtain
			\begin{align*}
				|j-m||T_j|^\frac{n-2}{2n} \leq \left(\I{T_j} |\phi_m(u_{\la,k,\e})|^\frac{2n}{n-2}\right)^\frac{n-2}{2n} \leq \left(\I{T_m} |\phi_m(u_{\la,k,\e})|^\frac{2n}{n-2}\right)^\frac{n-2}{2n}\leq M||\Om|^\frac{1}{p}|T_m|^{1-\frac{n-2}{2n}-\frac{1}{p}}
			\end{align*}
			which further implies
			\begin{align*}
				|T_j| \leq M \frac{|\Om|^\frac{2n}{(n-2)p}|T_m|^{\frac{2n}{n-2}\left(1-\frac{n-2}{2n}-\frac{1}{p}\right)}}{|j-m|^\frac{2n}{n-2}}.
			\end{align*}
			Since $p > \frac{n}{2}$, we have that
			\begin{align*}
				\frac{2n}{n-2}\left(1-\frac{n-2}{2n}-\frac{1}{p}\right) >1.
			\end{align*}
			Thus by \cite[Lemma B.1]{KS}, there exists $m_0$ such that $|A_m| =0$ for all $m \geq m_0$.
			This completes the proof.
			\QED
		\end{proof}
		\textbf{Proof of Proposition \ref{p3.2} $(2)$:} Let $\gamma \geq 3$ and $u_{\la,k,\e}$ be the weak solution of the problem $(\mc S_{\la,\e,k})$ in the sense that it satisfies \eqref{equ3.9}. Now from Lemma \ref{le3.4}, we know that $u_{\la,k,\e}^{\frac{\omega+1}{2}}$ is uniformly bounded in $H_0^1(\Om)$ with $\omega >\frac{\gamma-1}{2}$. Since $\gamma\geq 3$, we have $\omega > 1$. This together with the fact that for every compact subset $K \Subset \Om$ there exists $M =M(K)$ independent of $k$ such that $0< M \leq u_{\la,k,\e}(x)$ for $x \in K$, we get that $u_{\la,k,\e}$ is uniformly bounded in $H^1_{\text{loc}}(\Om)$. Precisely, 
		\begin{align*}
			\I{K}|\na u_{\la,k,\e}|^2 \leq M^{-(\omega-1)}\I{K} u_{\la,k,\e}^{(\omega-1)}|\na u_{\la,k,\e}|^2 dx \leq \frac{4 M^{-(\omega-1)}}{\omega+1}\I{K} | \na u_{\la,k,\e}^\frac{\omega+1}{2}|^2 \leq M_1, 
		\end{align*}
		where $M_1$ is independent of $k$. Then there exists a $u_\e \in H^1_{\text{loc}}(\Om)$ such that
		\begin{align}\label{eq3.10} u_{\la,k,\e} \rightharpoonup u_{\la,\e}, ~u_{\la,k,\e} \ra u_{\la,\e}  ~\text{in}~ L^r_{\text{loc}}(\Om)~\text{for}~ 1 \leq r <2^*~\text{ and a.e in}~ \Omega.
		\end{align}
 Now by using the weak convergence property we are able to pass to the limit in the left hand side of \eqref{equ3.9}, $i.e.$ for any $v \in H^1_{\text{loc}}(\Om)$ with $K=$ supp$(v) \Subset \Om$,
		\begin{equation}\label{equ3.11}
			\I \Om \na u_{\la,k,\e}\na vdx  \ra \I \Om \na u_{\la,\e}\na vdx ~\text{as}~k \ra \infty.
		\end{equation}
		Finally using the facts that $0\leq\chi_\e(u_{\la,k,\e}-a)\leq1$, $M_1(K) \leq u_{\la,1,\e} \leq u_{\la,k,\e}$ a.e. in $K$ and from Lebesgue dominated convergence theorem, we get
		\begin{equation}\label{equ3.12}
			\I\Om \frac{\chi_\e(u_{\la,k,\e}-a)}{(u_{\la,k,\e}+\frac{1}{k})^{\gamma}}v \ra \I\Om \frac{\chi_\e(u_{\la,\e}-a)}{u_{\la,\e}^{\gamma}}vdx.
		\end{equation}
		Passing to the limit in \eqref{equ3.9} and using \eqref{equ3.11} and \eqref{equ3.12}, we see that $u_\e$ is a weak solution of $(\mc S_{\la,\e})$. For uniqueness of the solution, using Theorem \ref{wcp}, it is sufficient to show that $(u_{\la,\e} -\nu)^+ \in H_0^1(\Om)$ for every $\nu >0$. From Lemma \ref{le3.4}, we have $u_{\la,k,\e}^\omega \in H_0^1(\Om)$, where $\omega >\frac{\gamma+1}{4}\geq 1$. Let $\varphi_m \in C_c^1(\Om)$ such that $\varphi_m$ converges to $u_{\la,k,\e}^\omega$ in $H_0^1(\Om)$ and set 
		\begin{align*}
			\psi_m := (\varphi_m^\frac{1}{\omega}-\nu)^+.
		\end{align*}
	Clearly $\psi_m$ is uniformly bounded in $H_0^1(\Om)$ and converges a.e. to $(u_{\la,k,\e}-\nu)^+$. Therefore we obtain that $(u_{\la,k,\e}-\nu)^+ \in H_0^1(\Om)$ and hence $(u_{\la,\e}-\nu)^+ \in H_0^1(\Om)$. This proves uniqueness. Next we prove that the map $\e \ra u_{\la,\e}$ is nonincreasing. For this, let $\e_1<\e_2$ and $u_{\la,\e_1},~u_{\la,\e_2}$ be the corresponding solutions. Arguing by contradiction, suppose  that there exists $F \subset \Om$ with positive measure such that $w =u_{\la,\e_1}-u_{\la,\e_2} <0$ a.e on $F$.  
		Using the facts that $u_{\la,\e_1}$, $u_{\la,\e_2}$ are solutions of $(\mc S_{\la,\e_1})$ and $(\mc S_{\la,\e_2})$, respectively, $\la>0$ and for $t_1 <t_2$ and $\e_1 <\e_2$, $\chi_{\e_1}(t_1-a) \geq \chi_{\e_2}(t_2 -a)$, by taking $v =w^-$, we get
		\begin{align*}
			\I\Om |\na w^-|^2 dx =-\la\I\Om \left(\chi_{\e_1}(u_{\la,\e_1}-a)u_{\la,\e_1}^{-\gamma}-\chi_{\e_2}(u_{\la,\e_2}-a)u_{\la,\e_2}^{-\gamma}\right)v \leq 0,
		\end{align*}
	and hence $w \geq 0$ a.e. in $\Om$, which yields a contradiction.
			Lastly we check the validity of \eqref{3.3} and \eqref{equ3.2}. From Lemma \ref{le3.4}, we have $u_{\la,k,\e}^{\omega}$ is uniformly bounded in $H_0^1(\Om)$ for $\omega >\frac{\gamma+1}{4}$. This implies the existence of $\psi \in H_0^1(\Om)$ such that $u_{\la,k,\e}^\omega \rightharpoonup \psi$ in $H_0^1(\Om)$ and $u_{\la,k,\e}^\omega \ra \psi$ un $L^r(\Om)$ for every $1 \leq r< 2^*$ a.e. in $\Om$. This together with \eqref{eq3.10} imply that $\psi=u_{\la,\e}^\omega$. Thus  we have
		\begin{align*}
			\|u_{\la,\e}^\omega\| \leq \liminf_{k \ra \infty} \|u_{\la,k,\e}^\omega\|^2 \leq M.
		\end{align*}
		From Lemma \ref{le3.4}, we know that $u_{\la,k,\e}$ is uniformly bounded in $L^\infty(\Om)$ both in $k$ and $\e$ and hence $u_{\la,\e}$ is uniformly bounded in $L^\infty(\Om)$.
		This completes the proof.\QED
		Lastly, we give the\\
		\textbf{Proof of Proposition \ref{p3.1}:} We divide the proof into two cases:\\
		\textbf{Case A}: $0<\gamma <3$. In this case, since $u_{\la,\e}$ is a solution of $(\mc S_{\la,\e})$, using Proposition \ref{p3.2}, Remark \ref{r2.4} and Hardy's inequality we find,
		\begin{align*}
			\|u_{\la,\e}\|^2 =\la \I\Om \chi_\e(u_{\la,\e}-a)u_{\la,\e}^{1-\gamma} \leq C \I\Om u_{\la,\e} \phi_{\gamma}^{-\gamma} \leq M \left( \I\Om |\na u_{\la,\e}|^2\right)^\frac12  = M\|u_{\la,\e}\|.
		\end{align*}
		Thus, $\{u_{\la,\e}\}_{\e}$ is a bounded sequence in $H_0^1(\Om)$. Let $u_{\la,\e} \rightharpoonup u_\lambda$ in $H_0^1(\Om)$ and $a.e$ in $\Om$. From the lower bound $u_{\la,\e} \geq M \phi_{\gamma}$, we obtain that $\{\chi_\e(u_{\la,\e}-a)u_{\la,\e}^{-\gamma}\}$ is a bounded sequence in $L^\infty_{\text{loc}}(\Om)$. Then using elliptic regularity theory, $\{u_{\la,\e}\}_\e$ is a bounded sequence in $C_{\text{loc}}^{\al}(\Om)$ for some $\al >0$ and hence, $u_{\la,\e} \ra u_\lambda$ uniformly on compact subsets of $\Om$. Let $\psi \in H_0^1(\Om)$ be arbitrary. Using Remark \ref{r2.4}, the estimate $\chi_\e(u_{\la,\e}-a)u_{\la,\e}^{-\gamma}\psi \leq M \phi_{\gamma}^{-\gamma}\psi$ and the weak convergence of $u_{\la,\e} \rightharpoonup u_\lambda$, we obtain that $u_\la$ solves $(\mc S_\la)$. \\
		\textbf{Case B:} $\gamma \geq 3$. The boundedness of the sequence $\{u_{\la,\e}\}_\e$ follows using the same arguments as in the proof of Proposition \ref{p3.2} $(2)$. Let $u_{\la,\e} \rightharpoonup u_\lambda$ in $H_0^1(\Om)$ and a.e. in $\Om$. Using the fact that for any compact set $K \Subset \Om$, $0<M(K) \leq u_{\la,\e}$ for every $\e >0$ and following the similar arguments of Case A, we conclude that $u_\la$ is a weak solution of $(\mc S_\la)$. \\
		In both cases, any solution to $(\mc S_\la)$ is a supersolution to  $(\mc S_{\la,\e})$. According to Theorem \ref{wcp}, $u_\lambda$ is the minimal solution to  $(\mc S_\la)$. Finally, the H\"{o}lder regularity results follow using the boundary behaviour and \cite[Theorem 1.2]{GL}}.\QED 	
		\section{Existence of a first solution for $(\mc P_\la)$}\label{s4}
		In this section, we establish the existence of first solution of the problem $(\mc P_\la)$. Here again we follow the regularising techniques as in the last section. Define
		\begin{equation*}\label{e4.1}
			\Lambda^a =\sup\{\la >0 : (\mc P_\la)~\text{has at least one solution}\}. 
		\end{equation*}
		We have
		\begin{Lemma}\label{l4.1}
			$0 < \Lambda^a <\infty.$
		\end{Lemma}
		\begin{proof}
			Let $(\mc P_\la)$ admit a solution $v_\la$.  Multiplying $(\mc P_\la)$ by $e_1$, we get
			\begin{align}\nonumber \label{4.2}
				\la_1 \I\Om v_\la e_1 dx=&\la \left(\I\Om \I \Om \frac{v_\la^{2_\mu^\ast}(y)v_\la^{2_\mu^\ast-1}(x)e_1(x)}{|x-y|^\mu}dxdy +\I\Om \chi_{\{v_\la<a\}}v_\la^{-\gamma}e_1(x)dx \right)\\
				&\geq	\la \left(\I\Om\left( \left(\frac{1}{d_\Om^\mu}\I\Om v_\la^{2_\mu^\ast}(y)dy\right)v_\la^{2_\mu^\ast-1}(x)+\chi_{\{v_\la<a\}}v_\la^{-\gamma}(x)\right)e_1(x)dx\right).
			\end{align}
		Now for any $m>0$ by noting the superlinear nature of the map $t \ra m t^{2_\mu^\ast-1}+\chi_{\{t<a\}}t^{-\gamma}$ at infinity, we guarantee the existence of a constant $ M=M(a)>0$ such that $ m t^{2_\mu^\ast-1}+\chi_{\{t<a\}}t^{-\gamma} >M t$. Employing this observation in \eqref{4.2}, we conclude that
		\begin{align*}
			\la_1\I\Om v_\la e_1dx \geq \la M \I\Om v_\la e_1 dx. 
		\end{align*}
			This implies $\Lambda^a <\infty$. Next we show that $0<\La^a$. For this, we consider the following singular problem without the jump discontinuity
			\begin{equation}\label{e4.2}
				\begin{cases}
					-\De u =\la \left(u^{-\gamma} +\left(\ds\I\Om  \frac{u^{2_\mu^\ast}(y)}{|x-y|^\mu}dy\right) u^{2_\mu^\ast-1}(x)\right)~\text{in}~\Om,
					u>0~\text{in}~\Om,~u=0~\text{on}~\partial \Om.
				\end{cases}
			\end{equation}
		It is well known (see Theorems $1.1$, $2.2$ and $2.5$ in \cite{CRT}) that there exists a unique $w_\la \in C_0(\overline{\Omega})\cap C^2(\Omega)$ solving the following singular problem for all $\la>0$:
			\begin{equation*}
			\left\{	-\De w = \la w^{-\gamma},~w >0~\text{in}~\Om,~v=0~\text{on}~\partial \Om.\right.
			\end{equation*}
		Also by Dini's Theorem, we have that $w_\la \ra 0$ uniformly in $\Om$ as $\la  \ra 0^+$. Clearly, $w_\la$ is a subsolution to \eqref{e4.2}. Next let $z \in H_0^1(\Om)$ solve
			\begin{equation*}
				\begin{cases}
					-\De   z =1,~z>0~\text{in}~\Om,~z=0~\text{in}~\partial \Om.
				\end{cases} 
			\end{equation*}
			Define $z_\la =w_\la + z$. Next we claim that there exists $\hat{\la} >0$ small such that for $\la <\hat{\la}$, $z_\la$ is a supersolution to \eqref{e4.2}. The choice of $\hat{\la}$ is such that $\hat{\la} (|w_\la+z|_\infty)^{22_\mu^\ast-1} \hat{M} \leq 1$, where $\hat{M}$ is such that $\left|\ds \I\Om \frac{dy}{|x-y|^\mu}\right| < \hat{M}$. Note that the choice of such  $\hat{M}$ is possible since $\Om$ is bounded. Then for $\la <\hat{\la}$, we have
			\begin{align*}
				-\De z_\la = \la w_\la^{-\gamma} +1 \geq \la\left(z_\la^{-\gamma}+(|w_\la+z|_\infty)^{22_\mu^\ast-1}\hat{M}\right)\geq\la\left(z_\la^{-\gamma}+ \I\Om \frac{z_\la^{2_\mu^\ast}(y)z_\la^{2_\mu^\ast-1}(x)}{|x-y|^\mu}dy\right).
			\end{align*}
			This completes the proof of the claim. Let $ \mc K_\la =\{u \in H_0^1(\Om): w_\la \leq u\leq z_\la~\text{in}~\Om\}$. Clearly, $\mc K_\la$ is a closed convex (hence weakly closed) set in $H_0^1(\Om)$. Now, we define the following iterative scheme for all $\la <\hat{\la}$:
			\begin{align*}
				\begin{cases}
					u_0 =w_\la,\\
					-\Delta u_k-\la u_k^{-\gamma} =\la \ds \I\Om \frac{u_{k-1}^{2_\mu^\ast}(y)u_{k-1}^{2_\mu^\ast -1}(x)}{|x-y|^\mu}dy,~u _k>0~\text{in}~\Om,\\
					u_k =0~\text{on}~\partial \Om,~k=1,2,3,\ldots
				\end{cases}
			\end{align*}
			The above scheme is well defined as we can solve for $u_n$ is closed convex set $\mc  K_\la$. Using the monotonicity of the operator $-\De u-\la u^{-\gamma}$, we have that the sequence $\{u_k\}_{k\in \mb N}$ is nondecreasing and $w_\la \leq u_k \leq z_\la$ for all $k$. In particular,
				$\{u_k\}_{k \in \mb N}$ is uniformly bounded in $C^\al(\bar{\Om})$. By the Ascoli-Arzela theorem, $u_k \ra 
				\bar u_\la \in C_0(\Om)$ as $k \ra  \infty$ and $w_\la \leq \bar u_\la \leq z_\la$. Now following the arguments as in the proof of Proposition \ref{p3.2}, we conclude that $\bar u_\la$ is a solution of \eqref{e4.2}. Noting that $|\bar u_\la|_\infty \ra 0$ for $\la \ra 0$, $\bar u_\la$ solves $(\mc P_\la)$ for small $\la$ and hence $\La^a >0$.   \QED
		\end{proof}
		Now we consider the following perturbed regular problem associated to $(\mc P_\la)$,
		\begin{align*}
			(\mc P_{\la,\e})	\begin{cases}	-\De u = \la \left( \ds\I\Om \frac{u^{2_\mu^\ast}(y)u^{2^\ast-1}(x)}{|x-y|^\mu}dy+\chi_\e(u-a)u^{-\gamma}\right),\\
				u \equiv 0~\text{on}~\partial\Om,~u> 0~\text{in}~\Om,
			\end{cases}
		\end{align*}
		where $\chi_\e$ is defined as in \eqref{eq3.2}. The formal energy functional $J_{\la,\e}$ associated to $(\mc P_{\la,\e})$ is defined as
		\begin{align*}
			J_{\la,\e}(u)= \frac12 \|u\|^2 -\la \I{\Om}H_\e(u)-\frac{\la}{22_\mu^*} \I\Om\I{\Om}\frac{|u|^{2_\mu^\ast}(y)|u|^{2^\ast}(x)}{|x-y|^\mu}dxdy,
		\end{align*}
		where $H_\e$ is defined as in \eqref{3.4}.\\
		We now have the following existence result:
		\begin{Lemma}\label{l4.2}
		$(\mc P_{\la})$ admits a solution $v_{\la}$ for all $\la \in (0,\La^a)$. Moreover, $v_\la \sim \phi_\la$ and satisfies the Sobolev regularity result as stated in Theorem \ref{tmr}.
		\end{Lemma}
		\begin{proof}
			We first show that given any $0<\e<\e_0(a)$, where $\e_0$ is obtained in Proposition \ref{p3.2} and $\la \in (0,\La^a)$, the approximating problem $(\mc P_{\la,\e})$ admits a solution $v_{\la,\e}$.
			Let $u_{\la,\e}$ be the solution of $(\mc S_{\la,\e})$ as obtained in Proposition \ref{p3.2}.
 Given any $\la \in (0,\La^a)$, there exists $\bar{\la} >\la$ such that $(\mc P_{\bar{\la}})$ admits a solution $\bar{v}$ and by the definition of $\chi_\e$, $\bar{v}$ is a supersolution of $(\mc P_{\la,\e})$. Then by Theorem \ref{wcp}, we see that $u_{\la,\e} \leq \bar{v}$.
  Now the existence of a solution $v_{\la,\e}$ of $(\mc P_{\la,\e})$ is obtained as a local minimizer of $J_{\la,\e}$ over the convex set $\mc M_\e =\{u \in H_0^1(\Om) : u_{\la,\e} \leq u\leq \bar{v}\}$. As $ v_{\la,\e}$ solves $(\mc P_{\la,\e})$, it is easy to check that $\{v_{\la,\e}\}_\e$ is bounded in $H_0^1(\Om)$ (thanks to $\bar{v}\in L^\infty(\Omega)$) and hence weakly converges to some $v_\la \in H_0^1(\Om)$. Then following the convergence arguments of Proposition \ref{p3.1}, we conclude that $v_\la$ is a solution of $(\mc P_\la)$.\\
  Now since $|v_{\la,\e}|_\infty \leq M$, where $M$ is independent of $\e$ (because that same is true for $u_{\la,\e}$ by Proposition \ref{p3.2}), we see that $u_{\la,\e} \leq v_{\la,\e} \leq w$, where $w$ is the solution of 
  	\begin{align*}
  		-\De w =\la w^{-\gamma}+ K,~w>0~\text{in}~\Om,~w=0~\text{on}~\partial \Om,
  	\end{align*}
  for some appropriate $K$. Also since $u_{\la,\e} \sim \phi_\la$ and $w \sim \phi_\la$, we conclude that $v_{\la,\e} \sim \phi_\la$ for every $\e$ and hence $v_\la \sim \phi_\gamma$.
  	 Finally, the Sobolev regularity results for $v_{\la,\e}$ follows on the similar lines of the proof of Lemma \ref{le3.4} item $v)$ using the fact that $v_{\la,\e} \in L^\infty(\Om)$ independent of $\e$ and for $v_\la$ using the arguments as in the proof of Proposition \ref{p3.2} $(2)$. \QED
		\end{proof}
		Next following the proof of Proposition $3.1$ and Lemma $3.4$ of \cite{DPST}, we have the following lemma:
		\begin{Lemma}
			For any $0<\la<\La^a$ and $0<\mu \leq\min\{n,4\}$, $ J_\la(v_\la) =\min_{v \in \mc M_0} J_\la(v)$, where $\mc M_0 =\{u \in H_0^1(\Om):u_\la \leq u\leq \bar{v}\}$ where $u_\lambda$ is as in Proposition \ref{p3.1}.
		\end{Lemma}
		Now we claim that $v_\la$ is a local minimum of $\mc J_\la$ in $H_0^1(\Om)$. We have
		\begin{Theorem}
			Let $a>0$ and $0<\mu \leq \min\{n,4\}$. Then for $\la \in (0,\La^a)$, $v_\la$ is a local minimum of $J_\la$ in $H_0^1(\Om)$.
		\end{Theorem}
		\begin{proof}
			We assume that $v_\la$ is not a local minimum of $J_\la$ in $H_0^1(\Om)$ and derive a contradiction. Let $\{v_k\} \subset H_0^1(\Om)$ be such that $v_k \ra v_\la$ in $H_0^1(\Om)$ and $J_\la(v_k) < J_\la(v_\la)$. For $\underline{v} =u_\la$ and solution $\overline{v}$ of $(\mc P_{\bar{\la}})$ where $0<\la<\bar{\la}<\La^a$, define $z_k =\max\{\underline{v},\min\{v_k,\overline{v}\}\}$, $\overline{w}_k =(v_k-\overline{v})^+$, $\underline{w}_k =(v_k -\underline{v})^-$, $\overline{A}_k =\text{supp}(\overline{w}_k)$ and $\underline{A}_k =\text{supp}(\underline{w}_k)$.\\
			\textbf{Claim A}: $|\overline{A}_k|$, $|\underline{A}_k|$ and $\| \overline{w}_k\| \ra 0$ as $k \ra \infty$.\\
			The proof of the claim can be proved on the similar lines of the proof of Theorem $2.2$ of \cite{DPST1} and hence omitted. Now note that $z_k \in \mc M_0 =\{u \in H^1_0(\Om):\underline{v} \leq u \leq \overline{v}\}$ and $v_k =z_k-\underline{w}_k +\overline{w}_k$.  Now
			 \begin{align}\nonumber
				J_\la (v_k) =& J_\la(z_k) + \frac12 \I {\overline{A}_k} \left(|\na v_k|^2-|\na \overline{v}|^2\right)dx +\frac12 \I {\underline{A}_k} \left(|\na v_k|^2-|\na \underline{v}|^2\right)dx -\la  \I {\overline{A}_k} (H (v_k)-H(\overline{v}))dx\\ \nonumber
				 &- \la \I {\underline{A}_k} (H (v_k)-H(\underline{v}))dx -\frac{\la}{22_\mu^\ast} \I\Om \I\Om \frac{v_k^{2_\mu^\ast}(x)v_k^{2_\mu^\ast}(y -z_k^{2_\mu^\ast}(x)z_k^{2_\mu^\ast}(y)}{|x-y|^\mu}dxdy\\ \nonumber
				\end{align}
			\begin{align}\label{4.6}\nonumber
				 =&J_\la(z_k) +\frac12\I{\overline{A}_k} |\na \overline{z}_k|^2 dx + \I {\overline{A}_k} \na \overline{v} \na \overline{z}_kdx+\frac12\I{\underline{A}_k} |\na \underline{z}_k|^2 dx + \I {\underline{A}_k} \na \underline{v} \na \underline{z}_kdx\\ \nonumber
				 &-\la  \I {\overline{A}_k} (H (\overline{v}+\overline{z}_k)-H(\overline{v}))dx- -\frac{\la}{22_\mu^\ast} \I\Om \I\Om \frac{(v_k^{2_\mu^\ast}(x)-z_k^{2_\mu^\ast}(x))v^{2_\mu^\ast}_k(y)}{|x-y|^\mu}dxdy\\ 
				 &- \la \I {\underline{A}_k} (H (\underline{v}_k-\underline{v})-H(\underline{v}))dx
				 -\frac{\la}{22_\mu^\ast} \I\Om \I\Om \frac{(v_k^{2_\mu^\ast}(x)-z_k^{2_\mu^\ast}(x))z^{2_\mu^\ast}_k(y)}{|x-y|^\mu}dxdy.
			\end{align}
		Employing the facts that $\underline{v}$ and $\overline{v}$ are respectively the sub- and supersolutions of $(\mc P_\la)$ we obtain from \eqref{4.6}
		\begin{align*}
			J_\la(v_k) \geq J_\la(z_k) + I_k + J_k,
		\end{align*}
	where 
	\begin{align*}
		I_k=& \frac{1}{2}\I{\overline{A}_k} |\na \overline{w}_k|^2 dx +\la  \I {\overline{A}_k} \left( \chi_{\{\overline{v}<a\}}\overline{v}^{-\gamma}\overline{w}_k-\left(H (\overline{v}+\overline{w}_k)-H(\overline{v})\right)\right)dx\\
		 &+\frac{\la}{2}\I \Om \I {\overline{A}_k} \frac{\overline{v}^{2_\mu^\ast}(y)\overline{v}^{2_\mu^\ast-1}(x)\overline{w}_k(x)-\frac{1}{2_\mu^\ast}\left((\overline{v}+\overline{w}_k)^{2_\mu^\ast}(x)-\overline{v}^{2_\mu^\ast}(x)\right)\left(v_k^{2_\mu^\ast}(y)+z_k^{2_\mu^\ast}(y)\right)}{|x-y|^\mu}dxdy
	\end{align*}
and 
\begin{align*}
		J_k=& \frac{1}{2}\I{\underline{A}_k} |\na \underline{w}_k|^2 dx -\la  \I {\underline{A}_k} \left( \left(\chi_{\{\underline{v}<a\}}\underline{v}^{-\gamma}\underline{w}_k+H (\underline{v}-\underline{w}_k)\right)-H(\underline{v})\right)dx\\
	&-\frac{\la}{2}\I \Om \I {\underline{A}_k} \frac{\underline{v}^{2_\mu^\ast}(y)\underline{v}^{2_\mu^\ast-1}(x)\underline{w}_k(x)-\frac{1}{2_\mu^\ast}\left((\underline{v}-\underline{w}_k)^{2_\mu^\ast}(x)-\underline{v}^{2_\mu^\ast}(x)\right)\left(w_k^{2_\mu^\ast}(y)+z_k^{2_\mu^\ast}(y)\right)}{|x-y|^\mu}dxdy.
\end{align*}
Now we claim that $I_k,~J_k \geq 0$ for large $k$ which is a contradiction to our assumption that $J_\la(v_k) < J_\la(v_\la)$ for all $k$. We only show that $I_k\geq 0$. The case of $J_k \geq 0$ runs in a similar fashion.\\
 Dividing $\overline{A}_k$ into three subdomains, viz, $\overline{A}_k \cap \{x \in \Om: a <\overline{v}(x)\}$, $\overline{A}_k\cap \{x \in \Om: \overline{v}(x)\leq a \leq (\overline{v}+\overline{w}_k)(x)\}$ and $\overline{A}_k\cap \{x \in \Om: (\overline{v}+\overline{w}_k)(x)<a\}$, one can check that the second integral on the right hand side of $I_k$ is nonnegative. Now using the fact that $z_k \leq \overline{v}$ and the mean value theorem, we obtain, for some $\theta =\theta(x) \in (0,1)$  that
\begin{align*}
	I_{k,1}&=\I\Om \I {\overline{A}_k} \frac{\overline{v}^{2_\mu^\ast}(y)\overline{v}^{2_\mu^\ast-1}(x)\overline{w}_k(x)-\frac{1}{2_\mu^\ast}\left((\overline{v}+\overline{w}_k)^{2_\mu^\ast}(x)-\overline{v}^{2_\mu^\ast}(x)\right)z_k^{2_\mu^\ast}(y)}{|x-y|^\mu}dxdy\\
	 \geq &\I \Om \frac{\overline{v}^{2_\mu^\ast}(y)}{|x-y|^\mu}dy \left( \I{\overline{A}_k} \left(\overline{v}^{2_\mu^\ast-1}(x)\overline{w}_k(x)-\frac{1}{2_\mu^\ast}\left((\overline{v}+\overline{w}_k)^{2_\mu^\ast}(x)-\overline{v}^{2_\mu^\ast}(x)\right)\right)dx\right)\\
	=&- \I \Om \frac{\overline{v}^{2_\mu^\ast}(y)}{|x-y|^\mu}dy \I{\overline{A}_k}\left(\left(\overline{v} +\theta \overline{w}_k\right)^{2_\mu^\ast-1}(x)-\overline{v}^{2_\mu^\ast-1}(x)\right) \overline{w}_k(x)dx\\
\end{align*}
\begin{align*}
	 \geq& - \I \Om \frac{\overline{v}^{2_\mu^\ast}(y)}{|x-y|^\mu}dy \I{\overline{A}_K} (\overline{v} + \overline{w}_k)^{2_\mu^\ast-2}\overline{w}_k^2dx \geq -M \I\Om \I{\overline{A}_k} \frac{\overline{v}^{2_\mu^\ast}(y)\left(\overline{v}^{2_\mu^\ast-2}(x)+\overline{w}_k^{2_\mu^\ast-2}(x)\right)\overline{w}_k^2(x )}{|x-y|^\mu}dxdy.
\end{align*}
Now by the application of Hardy-Littlewood-Sobolev and H\"{o}lder's inequalities, we obtain
\begin{align}\label{4.7}
I_{k,1}\geq M_1|\overline{v}|_{2^\ast}^{2_\mu^\ast}\left(\left(\I{\overline{A}_k}\overline{v}^{2^\ast}dx\right)^\frac{2_\mu^\ast-2}{2_\mu^\ast}\|\overline{w}_k\|^2 + \|\overline{w}_k\|^{2_\mu^\ast}\right).
\end{align}
Also
\begin{align*}
	I_{k,2} =& \I\Om \I {\overline{A}_k} \frac{\overline{v}^{2_\mu^\ast}(y)\overline{v}^{2_\mu^\ast-1}(x)\overline{w}_k(x)}{|x-y|^\mu}dxdy-\frac{1}{2_\mu^\ast}\I{\Om\setminus \overline{A}_k} \I {\overline{A}_k}\frac{\left((\overline{v}+\overline{w}_k)^{2_\mu^\ast}(x)-\overline{v}^{2_\mu^\ast}(x)\right)\overline{v}^{2_\mu^\ast}(y)}{|x-y|^\mu}dxdy\\
	&-\frac{1}{2_\mu^\ast}\I{\overline{A}_k} \I {\overline{A}_k}\frac{\left((\overline{v}+\overline{w}_k)^{2_\mu^\ast}(x)-\overline{v}^{2_\mu^\ast}(x)\right)v_k^{2_\mu^\ast}(y)}{|x-y|^\mu}dxdy\\
   \geq&\I{\Om }\I {\overline{A}_k} \frac{\overline{v}^{2_\mu^\ast}(y)\left(\overline{v}^{2_\mu^\ast-1}(x)\overline{w}_k(x)-\frac{1}{2_\mu^\ast}\left((\overline{v}+\overline{w}_k)^{2_\mu^\ast}(x)-\overline{v}^{2_\mu^\ast}(x)\right)\right)}{|x-y|^\mu}dxdy\\
   	&-\frac{1}{2_\mu^\ast}\I{\overline{A}_k} \I {\overline{A}_k}\frac{\left((\overline{v}+\overline{w}_k)^{2_\mu^\ast}(x)-\overline{v}^{2_\mu^\ast}(x)\right)v_k^{2_\mu^\ast}(y)}{|x-y|^\mu}dxdy.
\end{align*}
Now the first integral in the last inequality can be estimated like $I_{k,1}$. In the following we will estimate the second integral. Again using the mean value theorem, Hardy-Littlewood-Sobolev and H\"{o}lder's inequalities, we have
\begin{align}\nonumber \label{4.8}
	\frac{1}{2_\mu^\ast}\I{\overline{A}_k} \I {\overline{A}_k}\frac{\left((\overline{v}+\overline{w}_k)^{2_\mu^\ast}(x)-\overline{v}^{2_\mu^\ast}(x)\right)v_k^{2_\mu^\ast}(y)}{|x-y|^\mu}dxdy
	 \leq&\I{\overline{A}_k} \I {\overline{A}_k}\frac{ (\overline{v} + \overline{w}_k)^{2_\mu^\ast-1}(x)\overline{w}_k(x)v_k^{2_\mu^\ast}(y)}{|x-y|^\mu}dxdy\\ \nonumber
	\leq &M \left( \|\overline{v}\|_{L^{2^\ast}(\overline{A}_k)}+\|\overline{w}_k\|_{L^{2^\ast}(\overline{A}_k)}\right)^\frac{(2_\mu^\ast-1)2^\ast}{2} \|\overline{w}_k\|^{22^\ast}\\
	=& o_k(1).
\end{align}
Now using \eqref{4.7}, \eqref{4.8} and Claim A, we deduce that $I_k \geq 0$. This completes the proof of the theorem. \QED
		\end{proof}
		\section{Existence of a second solution for $(\mc P_\la)$} \label{s5}
		This section is devoted to obtain a second solution for $(\mc P_\la)$ for $\la \in (0,\La^a)$. Here we restrict ourselves to the case $0<\gamma<3$. We obtain the second solution by translating the problem to the solution $v_\la$ obtained in the previous section. Precisely, we consider the following problem
		\begin{align*}
			(\hat{P}_\la) 
			\begin{cases}
				-\De w =\la \left(\chi_{\{w+v_\la <a\}}(w+v_\la)^{-\gamma}-\chi_{\{v_\la<a\}}v_\la^{-\gamma}\right)\\
				~~~~~~~~~~~~~+ \la \ds\left(\I\Om \frac{(w+v_\la)^{2_\mu^\ast}(y)(w+v_\la)^{2_\mu^\ast-1}(x)-v_\la^{2_\mu^\ast}(y)v_\la^{2_\mu^\ast-1}(x)}{|x-y|^\mu}dy\right)~\text{in}~\Om,\\
				w>0~\text{in}~\Om,~w=0~\text{on}~ \partial \Om. 
				\end{cases}
		\end{align*}
		Clearly, if $w_\la \in H_0^1(\Om)$ weakly solves $(\hat{P}_\la)$, then $w_\la+v_\la$ weakly solves $(\mc P_\la)$. Let us define, for $x \in \Om$,
		\begin{align*}
			&f(x,s)=\left(\chi_{\{s+v_\la <a\}}(s+v_\la)^{-\gamma}-\chi_{\{v_\la<a\}}v_\la^{-\gamma}\right)\chi_{\mb R^+}(s).
		\end{align*}
		Let $F(x,t) =\ds \I{0}^{t} f(x,s)ds$. Now the energy functional $\mc G_\la: H_0^1(\Om) \ra \mb R$ associated with $(\hat{P}_\la)$ is given as:
		\begin{align*}\label{e5.1}\nonumber
			\mc G_\la(w)=&\frac12 \|w\|^2-\la \I\Om F(x,w)dx - \frac{\la}{2 2_\mu^\ast} \I\Om\I\Om \frac{(w+v_\la)^{2_\mu^\ast}(y)(w+v_\la)^{2_\mu^\ast}(x)}{|x-y|^\mu}dxdy\\
			 &-\la \I\Om\I\Om \frac{v_\la^{2_\mu^\ast}(y)v_\la^{2_\mu^\ast-1}(x)w(x)}{|x-y|^\mu}dxdy.
		\end{align*}
		\begin{Proposition}\label{p5.1}
			The map $\mc G_\la$ is locally Lipschitz.
		\end{Proposition}
		\begin{proof}
			The proof is similar to the proof of \cite[Proposition 3.1]{DPST1} and hence omitted. \QED 
		\end{proof}
		\begin{Remark}
			Note that $J_\la (w^+ +v_\la) = J_\la(v_\la) +\mc G_\la(w) -\frac12 \|w^-\|^2$ for any $ w\in H_0^1(\Om)$. Therefore, since $v_\la$ is a local minimum of $J_\la$, it follows that $0$ is a local minimum of $\mc G_\la$ in $H_0^1(\Om)$-topology.
		\end{Remark}
		\begin{Definition}
			Let $ \Phi: H_0^1(\Om)\ra \mb  R$ be a locally Lipschitz map. The generalized derivative of $\Phi$ at $u$ in the direction of $v$ (denoted by $\Phi^0(u,v)$) is defined as:
			\begin{equation*}
				\Phi^0(u,v)= \limsup\limits_{h \ra 0, t \downarrow 0} \frac{\Phi(u +h +tv)-\Phi(u+h)}{t},~u,~v \in H_0^1(\Om).
			\end{equation*}
			We say that $u$ is `generalized' critical point of $ \Phi$ if $\Phi^0(u,v) \geq 0$ for all $v \in H_0^1(\Om)$.
		\end{Definition}
		\begin{Remark}\label{r5.4}
			From \cite[Definition 4.1]{DPST}, for $w \geq 0$ and $\psi \in H_0^1(\Om)$, we have the following inequality:
			\begin{align}\nonumber\label{e5.2}
				\mc G_\la^0(w,\psi)=& \I\Om \na (v_\la +w)\na\psi dx-\la\I\Om\I\Om\frac{(v_\la +w)^{2_\mu^*}(y)(v_\la +w)^{2_\mu^*-1}(x)}{|x-y|^\mu}\psi dxdy\\
				 &-\la \I\Om z^{\psi}(v_\la +w)^{-\gamma} \psi dx,
			\end{align}
			for some measurable function $ z^\psi \in [\chi_{\{v_\la+w<a\}},\chi_{\{v_\la+w\leq a\}}]$.
		\end{Remark}
	\begin{Remark}\label{r5.5}
		Suppose for some nontrivial nonnegative $w_\la \in H_0^1(\Om)$ we have $\mc G_\la(w_\la,\psi) \geq 0$ for all $\psi \in H_0^1(\Om)$, i.e., $w_\la$ is a generalized critical point of $\mc G_\la$. Then we claim that 
		\begin{align}\nonumber \label{e5.3}
			\la \I\Om& \frac{(v_\la+w_\la)^{2_\mu^\ast}(y)(v_\la+w_\la)^{2_\mu^\ast-1}(x)}{|x-y|^\mu}dy \leq -\De (v_\la+w_\la)\\
			 &\leq\la\left(\I\Om \frac{(v_\la+w_\la)^{2_\mu^\ast}(y)(v_\la+w_\la)^{2_\mu^\ast-1}(x)}{|x-y|^\mu}dy+(v_\la+w_\la)^{-\gamma}\right).
		\end{align}
	Indeed, since $w_\la \geq 0$ and $\mc G_\la^0(w_\la,\psi) \geq0$, using \eqref{e5.2}, we have for any $ \psi \in H_0^1(\Om)$
	\begin{align}\nonumber\label{e5.4}
		0 \leq \mc G_\la^0(w_\la,\psi) =&\I\Om \na (v_\la +w_\la)\na\psi dx-\la\I\Om\I\Om\frac{(v_\la +w_\la)^{2_\mu^*}(y)(v_\la +w_\la)^{2_\mu^*-1}(x)}{|x-y|^\mu}\psi dxdy\\
		&-\la \I\Om z^{\psi}(v_\la +w_\la)^{-\gamma} \psi dx.
	\end{align}
Taking $\psi \geq 0$ in \eqref{e5.4}, we have
	\begin{align*}
	\la\I\Om\I\Om\frac{(v_\la +w)^{2_\mu^*}(y)(v_\la +w)^{2_\mu^*-1}(x)}{|x-y|^\mu}\psi dxdy
		+\la \I\Om z^{\psi}(v_\la +w)^{-\gamma} \psi dx \leq  	\I\Om \na (v_\la +w)\na\psi dx.
	\end{align*}
Since $z^\psi \geq0$ and given that $\psi \geq0$, we have
\begin{align}\label{e5.5}
		\la\I\Om\I\Om\frac{(v_\la +w)^{2_\mu^*}(y)(v_\la +w)^{2_\mu^*-1}(x)}{|x-y|^\mu}\psi dxdy\leq  	\I\Om \na (v_\la +w)\na\psi dx.
\end{align}
Next let us consider $ \varphi\in H_0^1(\Om)$ which is nonpositive, so that $\psi =-\varphi \geq 0$. Again using \eqref{e5.4}, we have
	\begin{align*}
	\la\I\Om\I\Om\frac{(v_\la +w)^{2_\mu^*}(y)(v_\la +w)^{2_\mu^*-1}(x)}{|x-y|^\mu}\varphi dxdy
	+\la \I\Om z^{\varphi}(v_\la +w)^{-\gamma} \varphi dx \leq  	\I\Om \na (v_\la +w)\na\varphi dx.
\end{align*}
Multiplying by $-1$ on both sides and using the fact that $z^\varphi \in [0,1]$, we get
\begin{align*}
	\la\I\Om\I\Om\frac{(v_\la +w)^{2_\mu^*}(y)(v_\la +w)^{2_\mu^*-1}(x)\psi(x)}{|x-y|^\mu} dxdy
	+\la \I\Om(v_\la +w)^{-\gamma} \psi dx \geq  	\I\Om \na (v_\la +w)\na\psi dx.
\end{align*}
Since $\psi =-\varphi$ is any arbitrary nonnegative function in $H_0^1(\Om)$, the previous expression implies that
\begin{align}\label{e5.6}
	-\De (v_\la+w_\la)
	\leq\la\left(\I\Om \frac{(v_\la+w_\la)^{2_\mu^\ast}(y)(v_\la+w_\la)^{2_\mu^\ast-1}(x)}{|x-y|^\mu}dy+(v_\la+w_\la)^{-\gamma}\right)~\text{weakly}.
\end{align}
Combining \eqref{e5.5} and \eqref{e5.6}, we have the validity of \eqref{e5.3}. Hence the claim.
	\end{Remark}
Note that $-\De(v_\la+w_\la)$ is a positive distribution and hence it is given by a positive, regular
Radon measure say $\nu$. Then using \eqref{e5.3}, we can show that $\nu$ is absolutely continuous with
respect to the Lebesgue measure. Now by Radon Nikodyn theorem there exists a locally integrable function $g$ such that $-\De(v_\la+w_\la) = g$ and hence $g \in L^p_{\text{loc}}(\Om)$ for some $p > 1$. Now using Lemma $B.3$ of \cite{MiS} and elliptic regularity, we can conclude that $v_\la +w_\la \in W^{2,q}_{\text{loc}} (\Om)$
for all $q < \infty$ and for almost every $x \in \Om$, using \eqref{e5.3} we have $-\De(v_\la+w_\la) > 0$. In particular,
\begin{align}\label{e5.7}
-\De(v_\la+w_\la) > 0~\text{for a.e on}~ \{x \in \Om: (v_\la+ w_\la)(x) = a\}. 
\end{align}
On the other hand, we have $-\De(v_\la + w_\la) = 0$ a.e on the set $\{x\in \Om : (v_\la + w_\la)(x) = a\}.$ This contradicts \eqref{e5.7} unless the Lebesgue measure of the set $\{x\in \Om : (v_\la + w_\la)(x) = a\}$ is zero. Therefore $z^\psi =\chi_{\{v_\la +w_\la<a\}}$ a.e. in $\Om$ for any $\psi \in H_0^1(\Om)$ and hence $v_\la+w_\la$ is a second solution for $(\mc P_\la)$.\\
		Our next target is to show the existence of a generalized critical point for $\mc G_\la$ which gives us the second solution of $(\mc P_{\la})$. We will employ the Mountain Pass theorem and Ekeland variational principle to this end. We define $X^+ =\{u \in H_0^1(\Om):~u \geq 0~\text{a.e in}~\Om\}$. Since $0$ is a local minimum of $\mc G_\la$, there exists a $\kappa_0 >0$ such that $\mc G_{\la}(0) \leq \mc G_\la (u)$ for $\|u\| \leq \kappa_0$. Then the following two cases arise:
		\begin{enumerate}
			\item $ZA$ (Zero Altitude): $ \inf \{\mc G_\la (w) :\|w\| =\kappa,~w \in X^+\} =\mc G_\la(0)=0$ for all $\kappa \in (0,\kappa_0)$.
			\item $MP$ (Mountain Pass): There exists $\kappa_1 \in (0,\kappa_0)$ such that $ \inf\{\mc G_\la(w)=\kappa_1,~w \in X^+\} > \mc G_\la(0)$.
		\end{enumerate}
		\begin{Lemma}\label{l5.6}
			Let $ZA$ hold for some $\la \in (0,\La^a)$. Then there exists a nontrivial `generalized' critical  point $w_\la \in X^+$ for $\mc G_\la$.
		\end{Lemma}
		\begin{proof}
			Fix $\kappa \in (0,\kappa_0)$. Then there exists a sequence $\{v_k\}_{k \in \mb N} \subset X^+$ with $\|v_k\| =\kappa$ and $\mc G_\la (v_k) \leq \frac{1}{k}$. Fix $0<q<\frac12\min\{\kappa_0-\kappa,\kappa\}$ and define $ \mc R = \{w \in X^+ : \kappa -q \leq \|w\| \leq \kappa +q\}$. Note that $\mc R$ is closed and $\mc G_\la$ is Lipschitz continuous on $\mc R$ (in view of Proposition \ref{p5.1}). Thus by Ekeland's variational Principle, there exists $\{\ell_k\}_{k\in \mb N}\subset \mc R$ such that the following holds:
			\begin{enumerate}
				\item $\mc G_\la (\ell_k) \leq \mc G_\la(v_k) \leq \frac{1}{k},$
				\item $\|\ell_k -v_k\| \leq \frac{1}{k}$ and
				\item $\mc G_\la(\ell_k) \leq \mc G_\la(l) +\frac{1}{k}\|l-\ell_k \|$ for all $l \in \mc R$.
			\end{enumerate}
			We note that
			\begin{align}\label{e5.8}
				\kappa-\frac{1}{k} =\|v_k\| -\frac{1}{k} \leq \|\ell_k\| \leq \|v_k\|+\frac{1}{k} =\kappa +\frac{1}{k}.  
			\end{align}
			Therefore, for $\aleph \in X^+$ we can choose $\e >0$ sufficiently small such that $ \ell_k +\e(\aleph -\ell_k) \in \mc R$ for all large $k$. Then by item $3$ above, we get
			\begin{align*}
				\frac{\mc G_\la(\ell_k +\e(\aleph -\ell_k))-\mc G_\la(\ell_k)}{\e} \geq \frac{-1}{k} \|\aleph -\ell_k\|.
			\end{align*}
			Letting $\e \ra 0^+$, we conclude
			\begin{align*}
				\mc G_\la^0(\ell_k,\aleph -\ell_k) \geq -\frac{1}{k}\|\aleph-\ell_k\|~\text{for all}~\aleph \in X^+.
			\end{align*}
			From Remark \ref{r5.4}, for any $\aleph \in X^+$, there exists $z_k^{\aleph -\ell_k} \in [\chi_{\{v_\la+\ell_k<a\}},\chi_{\{v_\la+\ell_k\leq a\}}]$ such that
			\begin{align}\nonumber
				\I\Om \na (v_\la+\ell_k)\na(\aleph-\ell_k)dx &-\la \I\Om\I\Om \frac{(v_\la +\ell_k)^{2_\mu^\ast}(y)(v_\la +\ell_k)^{2_\mu^\ast-1}(x)(\aleph-\ell_k)(x)}{|x-y|^\mu}dxdy\\ \nonumber
			\end{align}
		\begin{align}\label{e5.9}
				&-\la \I\Om z_k^{\aleph-\ell_k}(v_\la +\ell_k)^{-\gamma}(\aleph-\ell_k)dx \geq -\frac{1}{k} \|\aleph -\ell_k\|.
			\end{align}
			Since $\{\ell_k\}_{k \in \mb N}$ is bounded in $H_0^1(\Om)$, we may assume $\ell_k \rightharpoonup w_\la \in X^+$ weakly in $H_0^1(\Om)$ as well as a.e. in $\Om$. In the following we show that $w_\la$ is a solution of $(\hat{P}_\la)$. For $\varphi \in C_c^\infty(\Om)$, set 
			\begin{align*}
				\varphi_{k,\e} =(\ell_k+\e \varphi)^-~\text{and}~\ell=\ell_k+\e \varphi +\varphi_{k,\e}=(l_k+\e \varphi)^+.
			\end{align*} 
		Thus $\ell \in X^+$ and with this choice of $\ell$, \eqref{e5.9} becomes
			\begin{align}\nonumber\label{eq5.10}
			\I\Om \na (v_\la+\ell_k)\na(\e\varphi+\varphi_{k,\e})dx &-\la \I\Om\I\Om \frac{(v_\la +\ell_k)^{2_\mu^\ast}(y)(v_\la +\ell_k)^{2_\mu^\ast-1}(x)(\e\varphi+\varphi_{k,\e})(x)}{|x-y|^\mu}dxdy\\
			&-\la \I\Om z_k^{\e\varphi+\varphi_{\e,k}}(v_\la +\ell_k)^{-\gamma}(\e\varphi+\varphi_{k,\e})dx \geq -\frac{1}{k} \|\e\varphi+\varphi_{k,\e}\|,
		\end{align}
		where now in view of the fact $\{x\in\Om:(\e\varphi+\varphi_{k,\e})(x) \leq 0\} =\{x \in\Om : \varphi(x) \leq 0\}$, we have	
		\begin{equation}\label{e5.11}
			z_k^{\e\varphi+\varphi_{k,\e}}=\chi_{\{v_\la+\ell_k<a\}}+\chi_{\{v_\la+\ell_k=a\}\cap\{\varphi\leq0\}}.
		\end{equation}
	For a fixed $\e$, we now let $k \ra \infty$ and show that we can pass to the required limits in each of the terms in \eqref{eq5.10}. We have $\varphi_{k,\e} \rightharpoonup (w_\la +\e\varphi)^-$ in $H_0^1(\Om)$. It can be shown as in \cite[Lemma 7.5.2]{GR} that
	\begin{align}\label{e5.12}
		\I\Om \na (v_\la+\ell_k)\na(\e\varphi+\varphi_{k,\e})dx\leq 	\I\Om \na (v_\la+w_\la)\na(\e\varphi+(w_\la+\e\varphi)^{-})dx +o_k(1).
	\end{align}
Clearly $z_k^{\e\varphi+\varphi_{k,\e}}$ is bounded in $\Om$ and hence $z_k^{\e\varphi+\varphi_{k,\e}}\ra \tilde{z}$ weak* in $L^\infty(\Om)$. Since
\begin{align*}
	(\e\varphi+ \varphi_{k,\e})(v_\la+\ell_k)^{-\gamma} \ra (\e\varphi+(w_\la+\e\varphi)^-)(v_\la+w_\la)^{-\gamma}~\text{in}~L^1(\Om),
\end{align*}
we conclude that
\begin{align}\label{e5.13}
	\I\Om z_k^{\e\varphi+\varphi_{k,\e}}(\e\varphi+ \varphi_{k,\e})(v_\la+\ell_k)^{-\gamma} \ra \I\Om \tilde{z}(\e\varphi+(w_\la+\e\varphi)^-)(v_\la+w_\la)^{-\gamma}.
	\end{align}
Now it is standard that
\begin{align}\label{e5.14}\nonumber
	\I\Om\I\Om& \frac{(v_\la +\ell_k)^{2_\mu^\ast}(y)(v_\la +\ell_k)^{2_\mu^\ast-1}(x)\varphi(x)}{|x-y|^\mu}dxdy \\
	&\ra \I\Om\I\Om \frac{(v_\la +w_\la)^{2_\mu^\ast}(y)(v_\la +w_\la)^{2_\mu^\ast-1}(x)\varphi(x)}{|x-y|^\mu}dxdy.
\end{align}
Again since $0 \leq \varphi_{k,\e} \leq \e |\varphi|$, we see that
\begin{align}\label{e5.15}\nonumber
	\I\Om\I\Om& \frac{(v_\la +\ell_k)^{2_\mu^\ast}(y)(v_\la +\ell_k)^{2_\mu^\ast-1}(x)\varphi_{k,\e}}{|x-y|^\mu}dxdy\\
	 &\ra \I\Om\I\Om \frac{(v_\la +w_\la)^{2_\mu^\ast}(y)(v_\la +w_\la)^{2_\mu^\ast-1}(x)(w_\la+\e\varphi)^-}{|x-y|^\mu}dxdy.
\end{align}
Combining \eqref{eq5.10}-\eqref{e5.15}, we conclude that
	\begin{align}\nonumber\label{e5.16}
	0\leq& \I\Om \na (v_\la+w_\la)\na(\e\varphi+(w_\la+\e\varphi)^-)dx-\la \I\Om \tilde{z}(v_\la +w_\la)^{-\gamma}(\e\varphi+(w_\la+\e\varphi)^-)dx \\
	 &-\la \I\Om\I\Om \frac{(v_\la +w_\la)^{2_\mu^\ast}(y)(v_\la +w_\la)^{2_\mu^\ast-1}(x)(\e\varphi+(w_\la+\e\varphi)^-)(x)}{|x-y|^\mu}dxdy.
\end{align}
Note that $0\leq\tilde{z}\leq1$ a.e. in $\Om$ and $\tilde{z}$ depends only upon $\varphi$. Rewriting the above equation using the fact that $v_\la$ solves $(\mc P_\la)$ in the weak sense, we obtain
\begin{align}\nonumber\label{e5.17}
	&\I\Om \na(v_\la +w_\la)\na\varphi dx -\la \I\Om\I\Om \frac{(v_\la +w_\la)^{2_\mu^\ast}(y)(v_\la +w_\la)^{2_\mu^\ast-1}(x)\varphi(x)}{|x-y|^\mu}dxdy
	-\la \I\Om \tilde{z}(v_\la +w_\la)^{-\gamma}\varphi dx\\ \nonumber
	\geq& -\frac{1}{\e}\I\Om \na w_\la \na(w_\la+\e\varphi)^-dx+\frac{\la}{\e} \I\Om(\tilde{z}(v_\la+w_\la)^{-\gamma}-\chi_{\{v_\la<a\}}v_\la^{-\gamma})(w_\la+\e\varphi)^-dx\\
	&+ \frac{\la}{\e}\I\Om\I\Om \frac{((v_\la+w_\la)^{2_\mu^\ast}(y)(v_\la+w_\la)^{2_\mu^\ast-1}(x)-v_\la^{2_\mu^\ast}(y)v_\la^{2_\mu^\ast-1}(x))(v_\la+\e\varphi)^-}{|x-y|^\mu}dxdy.
\end{align}
Let $\Om_\e =\{x\in\Om:w_\la+\e\varphi\leq0\}$. Note that $|\Om_\e|\ra0$ as $\e \ra 0$ and hence
\begin{align*}
	-\I\Om \na w_\la \na(w_\la+\e\varphi)^-dx= \I{\Om_\e}|\na w_\la|^2 +\e \I{\Om_\e}\na w_\la\na\varphi dx \geq o(\e).
\end{align*}
Note that the last term in the RHS of \eqref{e5.17} is nonnegative. Using the fact that $v_\la^{-\gamma}\varphi \in L^1(\Om)$, we see that
\begin{align*}
	\I\Om(\tilde{z}(v_\la+w_\la)^{-\gamma}-\chi_{\{v_\la<a\}}v_\la^{-\gamma})(w_\la+\e\varphi)^-dx \leq \I{\Om_\e}2 v_\la^{-\gamma}(w_\la+\e\varphi)^- \leq 2\e \I{\Om_\e}v_\la^{-\gamma}|\varphi| =o(\e).
\end{align*}
Letting $\e \ra0$ in \eqref{e5.17}, it can be seen that given any $\varphi \in H_0^1(\Om)$, there exists $\tilde{z}=\tilde{z}^\varphi$ with $0\leq \tilde{z}\leq1$ such that
\begin{align*}
0\leq&	\I\Om \na(v_\la +w_\la)\na\varphi dx -\la \I\Om\I\Om \frac{(v_\la +w_\la)^{2_\mu^\ast}(y)(v_\la +w_\la)^{2_\mu^\ast-1}(x)\varphi(x)}{|x-y|^\mu}dxdy\\
	&-\la \I\Om \tilde{z}(v_\la +w_\la)^{-\gamma}\varphi dx.
\end{align*}
Moreover since $\tilde{z}$ is obtained as the weak* limit of $z_k^{\e\varphi+\varphi_{k,\e}}$, we get that $\tilde{z} =1$ a.e. in $\{v_\la+w_\la <a\}$ and $\tilde{z} =0$ a.e. in $\{v_\la+w_\la >a\}$. Therefore from Remark \ref{r5.5}, $w_\la$ is a generalized critical point of $\mc G_\la$. It remains to show that $w_\la \not \equiv 0$. Note that if $\mc G_\la(w_\la) \ne 0$ we are done. So assume $\mc G_\la(w_\la)=0$. From \eqref{e5.8}, we see that $\|\ell_k\|\geq \ds\frac{\kappa}{2}$ for all large $k$. Thus it is sufficient to show that $\ell_k \ra w_\la$ strongly in $H_0^1(\Om)$. Taking $\aleph =w_\la$ in \eqref{e5.9}, we get
\begin{align}\nonumber\label{e5.18}
	\I\Om \na (v_\la+w_\la)&\na(w_\la-\ell_k)dx -\la \I\Om\I\Om \frac{(v_\la +\ell_k)^{2_\mu^\ast}(y)(v_\la +\ell_k)^{2_\mu^\ast-1}(x)(w_\la-\ell_k)(x)}{|x-y|^\mu}dxdy\\
	&-\la \I\Om z_k^{w_\la-\ell_k}(v_\la +\ell_k)^{-\gamma}(w_\la-\ell_k)dx+\frac{1}{k} \|w_\la -\ell_k\| \geq \|w_\la-\ell_k\|^2.
\end{align}
Now for any measurable set $E \subset \Om$, as $v_\la \geq M_1\phi_\gamma$ and $\ell_k \in X^+$, thanks to Hardy's inequality, we have
\begin{align}\label{5.19}\nonumber
	\I{E}z_k^{w_\la-\ell_k}|\ell_k -w_\la|(v_\la+\ell_k)^{-\gamma} \leq M\I{E}\frac{|\ell_k-w_\la|}{v_\la^\gamma}\leq& M \I{E}\frac{|\ell_k-w_\la|}{\de(x)^\frac{2\gamma}{1+\gamma}}\leq M \I{E}\frac{|\ell_k-w_\la|}{\de(x)}\de(x)^\frac{1-\gamma}{1+\gamma}\\
	\leq& M \|\ell_k-w_\la\| \|\de(x)^\frac{1-\gamma}{1+\gamma}\|_{L^2(E)}.
\end{align}
Since $\ell_k \ra w_\la$ pointwise a.e. in $\Om$, by Vitali's convergence theorem,
\begin{align}\label{e5.19}
	\I{E}z_k^{w_\la-\ell_k}|\ell_k -w_\la|(v_\la+\ell_k)^{-\gamma}\ra 0~\text{as}~k\ra \infty.
\end{align}
Also using Brezis-Lieb Lemma, we have
\begin{align}\nonumber\label{e5.20}
	\I\Om\I\Om& \frac{(v_\la +\ell_k)^{2_\mu^\ast}(v_\la +\ell_k)^{2^\ast-1}(w_\la-\ell_k)}{|x-y|^\mu}dxdy\\ \nonumber
	 =&	\I\Om\I\Om \frac{(v_\la +\ell_k)^{2^\ast}(v_\la +\ell_k)^{2^\ast-1}(w_\la+v_\la)}{|x-y|^\mu}dxdy
	 -\I\Om\I\Om \frac{(v_\la +\ell_k)^{2^\ast}(v_\la +\ell_k)^{2^\ast}}{|x-y|^\mu}dxdy\\
	 =&-\I\Om\I\Om \frac{(w_\la-\ell_k)^{2_\mu^\ast}(w_\la-\ell_k)^{2_\mu^\ast}}{|x-y|^\mu}+o_k(1).
\end{align}
Now using \eqref{e5.19} and \eqref{e5.20} in \eqref{e5.18}, we get
\begin{align}\label{e5.21}
	\|w_\la -\ell_k\|^2 -\la \I\Om\I\Om \frac{(w_\la-\ell_k)^{2_\mu^\ast}(w_\la-\ell_k)^{2_\mu^\ast}}{|x-y|^\mu}dxdy \leq o_k(1). 
\end{align}
Again taking $\aleph =2\ell_k$ in \eqref{e5.9} and using the fact that $v_\la$ solves $(\mc P_\la)$, we obtain
\begin{align}\nonumber\label{eq5.22}
	-\frac{1}{k}\|\ell_k\| \leq& \I\Om \na v_\la \na \ell_k dx+\I\Om|\na \ell_k|^2dx-\la\I\Om\I\Om \frac{(v_\la+\ell_k)^{2_\mu^\ast}(v_\la+\ell_k)^{2_\mu^\ast-1}\ell_k}{|x-y|^\mu}dxdy\\ \nonumber
	&-\la \I\Om z_k^{\ell_k}l_k(v_\la+\ell_k)^{-\gamma}dx\\ \nonumber
	=& \|\ell_k\|^2-\la\I\Om\I\Om \frac{((v_\la+\ell_k)^{2_\mu^\ast}(v_\la+\ell_k)^{2_\mu^\ast-1}-v_\la^{2_\mu^\ast}v_\la^{2_\mu^\ast-1})\ell_k}{|x-y|^\mu}dxdy\\ \nonumber
	&+\la\I\Om (\chi_{\{v_\la<a\}}v_\la^{-\gamma}-z_k^{\ell_k}(v_\la+\ell_k)^{-\gamma})\ell_kdx\\ \nonumber
	=& \|w_\la\|^2+\|\ell_k-w_\la\|^2-\la\I\Om\I\Om \frac{((v_\la+\ell_k)^{2_\mu^\ast}(v_\la+\ell_k)^{2_\mu^\ast-1}-v_\la^{2_\mu^\ast}v_\la^{2_\mu^\ast-1})\ell_k}{|x-y|^\mu}dxdy\\
	&+\la\I\Om (\chi_{\{v_\la<a\}}v_\la^{-\gamma}-z_k^{\ell_k}(v_\la+\ell_k)^{-\gamma})\ell_kdx+o_k(1).
\end{align}
Now as $w_\la$ solves $(\hat{P}_\la)$, we have
\begin{align*}
	\|w_\la\|^2 =\la\I\Om\I\Om& \frac{((v_\la+w_\la)^{2_\mu^\ast}(v_\la+w_\la)^{2_\mu^\ast-1}-v_\la^{2_\mu^\ast}v_\la^{2_\mu^\ast-1})w_\la}{|x-y|^\mu}dxdy\\
\end{align*}
\begin{align*}
		&+\la\I\Om (\chi_{\{v_\la+w_\la<a\}}(v_\la+w_\la)^{-\gamma}-\chi_{\{v_\la<a\}}v_\la^{-\gamma})w_\la dx.
\end{align*}
Using this identity in \eqref{eq5.22}, we get 
\begin{align}\nonumber\label{e5.24}
	-\frac{1}{k}&\|\ell_k\|
	\leq\|\ell_k-w_\la\|^2-\la\I\Om\I\Om \frac{((v_\la+\ell_k)^{2_\mu^\ast}(v_\la+\ell_k)^{2_\mu^\ast-1}-v_\la^{2_\mu^\ast}v_\la^{2_\mu^\ast-1})\ell_k}{|x-y|^\mu}dxdy\\ \nonumber
	&+\la\I\Om (\chi_{\{v_\la+w_\la<a\}}(v_\la+w_\la)^{-\gamma}-\chi_{\{v_\la<a\}}v_\la^{-\gamma})w_\la dx+\la\I\Om (\chi_{\{v_\la<a\}}v_\la^{-\gamma}-z_k^{\ell_k}(v_\la+\ell_k)^{-\gamma})\ell_kdx\\ 
	&+\la\I\Om\I\Om \frac{((v_\la+w_\la)^{2_\mu^\ast}(v_\la+w_\la)^{2_\mu^\ast-1}-v_\la^{2_\mu^\ast}v_\la^{2_\mu^\ast-1})w_\la}{|x-y|^\mu}dxdy
	+o_k(1).
	\end{align}
Now as $\ell_k \ra w_\la$ pointwise a.e. in $\Om$ and $|\{x\in \Om:(v_\la=w_\la)(x)=a\}|=0$, using estimates similar to \eqref{5.19} we have
\begin{align*}
	\I\Om (\chi_{\{v_\la+w_\la<a\}}(v_\la+w_\la)^{-\gamma}-\chi_{\{v_\la<a\}}v_\la^{-\gamma})w_\la dx+\I\Om (\chi_{\{v_\la<a\}}v_\la^{-\gamma}-z_k^{\ell_k}(v_\la+\ell_k)^{-\gamma})\ell_kdx =o_k(1).
\end{align*}
Using above estimate and Brezis-Lieb lemma, we obtain form \eqref{e5.24}
\begin{align}\label{e5.25}
	o_k(1) \leq \|\ell_k -w_\la\|^2- \la \I\Om\I\Om \frac{(w_\la-\ell_k)^{2_\mu^\ast}(y)(w_\la-\ell_k)^{2_\mu^\ast}(x)}{|x-y|^\mu}dxdy.
\end{align}
Also as $\mc G_\la(\ell_k) \leq \frac{1}{k}$, we have
\begin{align*}
 \mc G_\la(\ell_k)=&\frac12 \|\ell_k\|^2-\la \I\Om F(x,\ell_k)dx - \frac{\la}{2 2_\mu^\ast} \I\Om\I\Om \frac{(\ell_k+v_\la)^{2_\mu^\ast}(\ell_k+v_\la)^{2_\mu^\ast}}{|x-y|^\mu}dxdy\\
	&-\la \I\Om\I\Om \frac{v_\la^{2_\mu^\ast}v_\la^{2_\mu^\ast-1}\ell_k}{|x-y|^\mu}dxdy \leq \frac{1}{k}.
\end{align*}
From the fact that $\ell_k \rightharpoonup w_\la$ weakly in $H_0^1(\Om)$, this implies
\begin{align}\nonumber\label{e5.26}
	\frac12\|\ell_k -w_\la\|^2 -& \frac{\la}{22_\mu^\ast}\I\Om\I\Om \frac{(w_\la-\ell_k)^{2_\mu^\ast}(y)(w_\la-\ell_k)^{2_\mu^\ast}(x)}{|x-y|^\mu}dxdy +\mc G_\la(w_\la)+\la \I\Om F(x,w_\la)dx\\
	 &-\la \I\Om F(x,\ell_k)dx \leq o_k(1).
\end{align}
Now using the Hardy's inequality and Vitali's convergence theorem as in \eqref{5.19} one can check that $\ds \I\Om F(x,\ell_k)dx\ra \ds \I\Om F(x,w_\la)dx$ as $k \ra \infty$. Also as $\mc G_\la(w_\la)=0$, \eqref{e5.26} implies
\begin{align}\label{e5.27}
	\frac12\|\ell_k -w_\la\|^2 -\frac{\la}{22_\mu^\ast}\I\Om\I\Om \frac{(w_\la-\ell_k)^{2_\mu^\ast}(y)(w_\la-\ell_k)^{2_\mu^\ast}(x)}{|x-y|^\mu}dxdy \leq o_k(1).
\end{align}
Now from \eqref{e5.21}, \eqref{e5.25} and \eqref{e5.27} we get $\left(\frac{1}{2}-\frac{1}{22_\mu^\ast}\right)\|\ell_k-w_\la\| \leq o_k(1)$ and hence $\ell_k \ra w_\la$ in $H_0^1(\Om)$. \QED
		\end{proof}		 		
	Next we consider the case $(MP)$. Here to deal with the critical nonlinearity, we use the following Talenti functions to study the critical level:
	\begin{align*}
		V_\e(x) =S^\frac{(n-\mu)(2-n)}{4(n-\mu+2)} (C(n,\mu))^\frac{2-n}{2(n-\mu+2)}\left(\frac{\e}{\e^2+|x|^2}\right)^\frac{n-2}{2},~0<\e<1.
	\end{align*}
Fix any $y \in \Om_a =\{x\in \Om:v_\la(x)<a\}$. Let $r >0$ such that $B_{4r}(y) \subset \Om$. Now define $\psi \in C_c^\infty(\Om)$ such that $0 \leq \eta \leq 1$ in $\mb R^n$, $\eta \equiv 1$ in $B_r(y)$ and $\eta \equiv 0$ in $\mb R^n \setminus B_{2r}(y)$. For each $\e >0$ and $x \in \mb R^n$, we define $w_\e(x) =\psi(x)V_\e(x)$. In the following, we set the notation
\begin{align*}
	\|u\|_{HL}^{22_\mu^\ast} =\I\Om\I\Om \frac{|u(y)|^{2_\mu^\ast}|u(x)|^{2_\mu^\ast}}{|x-y|^\mu}dxdy.
\end{align*}
	\begin{Proposition}\label{p4.11}
		Let $n \geq 3$, $0<\mu <n$ then the following holds:
		\begin{enumerate}[label =(\roman*)]
			\item  $\| w_\e\| ^2 \leq S_{H,L}^\frac{2n-\mu}{n-\mu+2}+O (\e^{n-2}).$
			\item $\|w_\e\|_{HL}^{22_\mu^*} \leq S_{H,L}^\frac{2n-\mu}{n-\mu+2} +O(\e^n)$.
			\item $\|w_\e\|_{HL}^{22_\mu^*} \geq  S_{H,L}^\frac{2n-\mu}{n-\mu+2} -O(\e^n)$.
		\end{enumerate}
	\end{Proposition}
	\begin{proof}
	For proof of part $(i)$, we refer to \cite[Lemma 1.46]{W}. For $(ii)$ and $(iii)$, see \cite[Proposition 2.8]{GMS1}. \QED 
\end{proof}	
\begin{Lemma}\label{l4.12}
	The following holds:
	\begin{enumerate}[label=(\roman*)]
		\item If $ \mu < \min\{4,n\}$ then for all $\kappa <1$,
		\begin{align*}
			\| v_\la + t w_\e\|_{HL}^{22_\mu^*} \geq& \|v_\la\|_{HL}^{22_\mu^*}+\|tw_\e\|_{HL}^{22_\mu^*} +\widetilde{M} t^{22_\mu^* -1}\I\Om\I\Om \frac{w_\e^{2_\mu^*}(y)w_\e^{2_\mu^*-1}(x)v_\la(x)}{|x-y|^\mu}dxdy\\
			+& 22_\mu^*t \I\Om\I\Om \frac{v_\la^{2_\mu^*}(y)v_\la^{2_\mu^*-1}(x) w_\e(x)}{|x-y|^\mu}dxdy -O(\e^{(\frac{2n-\mu}{4})\kappa}).
		\end{align*}
		\item There exists a constant $T_0 >0$ such that $\ds\I\Om\I\Om \frac{w_\e^{2_\mu^*}(y)w_\e^{2_\mu^*-1}(x) v_\la(x)}{|x-y|^\mu}dxdy \geq \ds \widetilde{C} T_0 \e^{\frac{n-2}{2}}$.
	\end{enumerate}
\end{Lemma}
\begin{proof}
	For a proof, see the proof of  \cite[Lemma 4.2]{GS}. \QED
\end{proof}
Next we have the following lemma:
\begin{Lemma}\label{l5.9}
	There exists $\e_0$ and $R_0 \geq 1$ such that
	\begin{enumerate}[label=(\roman*)]
		\item $\mc G_\la (Rw_\e)<\mc G_\la(0)=0$ for all $\e \in (0,\e_0)$ and $R \geq R_0$.
		\item $\mc G_\la(tR_0w_\e) <\ds\frac12 \left(\frac{n-\mu+2}{2n-\mu}\right) \left(\frac{S_{H,L}^\frac{2n-\mu}{n-\mu+2}}{\la^\frac{n-2}{n-\mu+2}}\right)$ for all $t \in (0,1]$ and  $\e \in (0,\e_0)$.
	\end{enumerate}
\end{Lemma}
\begin{proof}
	Noting that for $w \in X^+$, $J_\la(v_\la +w) =J_\la(v_\la)+\mc G_\la(w)$, this is equivalent to show that
	\begin{enumerate}[label=\roman*)]
		\item $ J_\la (v_\la+Rw_\e)< J_\la(v_\la)$ for all $\e \in (0,\e_0)$ and $R \geq R_0$.
		\item $ J_\la(v_\la+tR_0w_\e) <J_\la(v_\la)+\ds\frac12 \left(\frac{n-\mu+2}{2n-\mu}\right) \left(\frac{S_{H,L}^\frac{2n-\mu}{n-\mu+2}}{\la^\frac{n-2}{n-\mu+2}}\right)$ for all $t \in (0,1]$ and  $\e \in (0,\e_0)$.
	\end{enumerate}
Now using the fact that $v_\la$ solves $(\mc P_\la)$, first we estimate $J_\la (v_\la +t R w_\e)$ as follows
\begin{align*}
	J_\la(v_\la+tRw_\e) =&\frac12\I\Om |\na(v_\la +t R w_\e)|^2dx -\la \I\Om F(v_\la +tRw_\e)dx -\frac{\la}{22_\mu^\ast}\|v_\la+tRw_\e\|_{HL}^{22_\mu^\ast}\\
	=&\frac12\I\Om |\na v_\la|^2dx +\frac{R^2t^2}{2}\I\Om|\na w_\e|^2dx +tR \I\Om \na v_\la \na w_\e dx\\
	&-\la \I\Om F(v_\la +tRw_\e)dx-\frac{\la}{22_\mu^\ast}\|v_\la+tRw_\e\|_{HL}^{22_\mu^\ast}\\
	=&\frac12\I\Om |\na v_\la|^2dx +\frac{R^2t^2}{2}\I\Om|\na w_\e|^2dx +\la tR \I\Om \chi_{\{v_\la<a\}}v_\la^{-\gamma}w_\e dx-\frac{\la}{22_\mu^\ast}\|v_\la+tRw_\e\|_{HL}^{22_\mu^\ast}\\
	&+\la tR \I\Om\I\Om \frac{v_\la^{2_\mu^\ast}(y)v_\la^{2_\mu^\ast-1}(x)w_\e(x)}{|x-y|^\mu}dxdy -\la \I\Om F(v_\la +tRw_\e)dx\\
	\leq & J_\la(v_\la) +\frac{1}{2}\|tRw_\e\|^2 -\frac{\la}{22_\mu^\ast}\|tR w_\e\|_{HL}^{22_\mu^\ast}-\frac{\la\widetilde{M} t^{22_\mu^* -1}}{22_\mu^\ast}\I\Om\I\Om \frac{w_\e^{2_\mu^*}(y)w_\e^{2_\mu^*-1}(x)v_\la(x)}{|x-y|^\mu}dxdy\\
	&+\la \I\Om \left(\chi_{\{v_\la<a\}}v_\la^{-\gamma}w_\e +F(v_\la)- F(v_\la +tRw_\e) \right) +O(\e^{(\frac{2n-\mu}{4})\kappa}).
\end{align*}
Using Proposition \ref{p4.11} and Lemma \ref{l4.12}, we obtain
\begin{align}\nonumber \label{e5.28}
	J_\la&(v_\la+ t Rw_\e)\\ \nonumber
	 \leq& J_\la(v_\la)+ \frac{t^2R^2}{2}\left(S_{H,L}^\frac{2n-\mu}{n-\mu+2}+O(\e^{\frac{n-2}{2}})\right)-\frac{\la t^{22_\mu^*}R^{22_\mu^*}}{2 2_\mu^*} \left(S_{H,L}^\frac{2n-\mu}{n-\mu+2}-O(\e^n)\right)+O(\e^{(\frac{2n-\mu}{4})\kappa})\\
	&+\la \I\Om \left(\chi_{\{v_\la<a\}}v_\la^{-\gamma}w_\e +F(v_\la)- F(v_\la +tRw_\e) \right)-\frac{\la\widetilde{M}t^{22_\mu^*-1}}{22_\mu^*}T_0 \e^\frac{n-2}{2}.
\end{align}
Finally by estimating the last integral in \eqref{e5.28} using the similar lines of \cite[Lemma 3.2]{DPST1} (Page no. 1671-1672), we conclude by taking $\kappa =\frac{2}{2_\mu^\ast}$ that
\begin{align*}
	J_\la&(v_\la+ t Rw_\e)\\
	 \leq& J_\la(v_\la)+ \frac{t^2R^2}{2}\left(S_{H,L}^\frac{2n-\mu}{n-\mu+2}+O(\e^{\frac{n-2}{2}})\right)-\frac{\la t^{22_\mu^*}R^{22_\mu^*}}{2 2_\mu^*} \left(S_{H,L}^\frac{2n-\mu}{n-\mu+2}-O(\e^n)\right)+O(\e^{(\frac{2n-\mu}{4})\kappa})\\
	&-\frac{\la\widetilde{M}t^{22_\mu^*-1}}{22_\mu^*}T_0 \e^\frac{n-2}{2}+o(\e^\frac{n-2}{2}).
\end{align*}
Now the Lemma follows using the arguments as in \cite[Lemma 6.4]{GGS1}. \QED
	\end{proof}
	\begin{Lemma}\label{l5.10}
		Let $(MP)$ hold. Then there exists a solution $w_\la \in X^+$ of $(\hat{P}_\la)$ and hence a second solution of $(\mc P_\la)$.
	\end{Lemma}
\begin{proof}
	Define a complete metric space $(M,d)$ as 
	\begin{align*}
		M=\{\zeta \in C([0,1],&X^+):~\zeta(0)=0,~\|\zeta(1)\| >\kappa_1,~\mc G_\la(\zeta(1))<0\},\\
		&d(\zeta,\eta) =\max_{t \in [0,1]}\|\zeta(t)-\eta(t)\|.
		\end{align*}
	From $(i)$ of Lemma \ref{l5.9}, if $R$ is chosen large, it is clear that $M$ is non-empty. Let 
	$$ c_0 =\inf\limits_{\zeta \in M}\max\limits_{t \in [0,1]} \mc G_\la(\zeta(t)).$$
	 Then $(ii)$ of Lemma \ref{l5.9} and $(MP)$ implies that
	\begin{align}\label{5.28}
		0<c_0<\ds\frac12 \left(\frac{n-\mu+2}{2n-\mu}\right) \left(\frac{S_{H,L}^\frac{2n-\mu}{n-\mu+2}}{\la^\frac{n-2}{n-\mu+2}}\right).
	\end{align}
Define 
\begin{align*}
	\Phi(\zeta) =\max_{t \in [0,1]}\mc G_\la(\zeta(t)),~\zeta \in M.
\end{align*}
By applying Ekeland's variational principle to the above functional we get a sequence $\{\zeta_k\}_{k \in \mb N} \subset M$ such that
\begin{enumerate}[label=(\roman*)]
	\item $\max\limits_{t \in [0,1]} \mc G_\la (\zeta_k(t)) <c_0 +\frac{1}{k}$.
	\item $\max\limits_{t \in [0,1]} \mc G_\la (\zeta_k(t)) \leq \max\limits_{t \in [0,1]} \mc G_\la(\zeta(t)) +\frac{1}{k} \max\limits_{t \in [0,1]} \|\zeta(t)-\zeta_k(t)\|$ for all $\zeta \in M$.
\end{enumerate}
Setting $ \Gamma_k =\{t \in [0,1]: \mc G_\la(\zeta_k(t))=\max\limits_{s \in [0,1]} \mc G_\la (\zeta_k(s))\}$ we obtain by arguing as in \cite[Page no. 659]{BT} $t_k \in \Gamma_k$ such that for $w_k= \zeta_k(t_k)$ and $\aleph \in X^+$, we have
\begin{align}\label{e5.29}
	\mc G_\la^0 \left(w_k,\frac{\aleph-w_k}{\max\{1,\|\aleph-w_k\|\}}\right) \geq -\frac{1}{k}
\end{align}
and 
\begin{align}\label{e5.30}
	\mc G_\la(w_k) \ra c_0~\text{as}~k \ra \infty.
\end{align}
From \eqref{e5.30} and using the fact that $F(w_k)\leq 0$, we have
\begin{align}\label{e5.31}
	c_0+o_k(1)\geq&\frac12 \|w_k\|^2- \frac{\la}{2 2_\mu^\ast} \|w_k+v_\la\|_{HL}^{22_\mu^\ast}
	-\la \I\Om\I\Om \frac{v_\la^{2_\mu^\ast}(y)v_\la^{2_\mu^\ast-1}(x)w_k(x)}{|x-y|^\mu}dxdy.
\end{align}
Again substituting $\aleph =2w_k +v_\la$ in \eqref{e5.29}, by Remark \ref{r5.4} we obtain
\begin{align*}
	z_k^{2w_k+v_\la}=\tilde{z}_k \in [\chi_{\{w_k+v_\la<a\}},\chi_{\{w_k+v_\la\leq a\}}]
\end{align*}
such that
\begin{align}\label{e5.32}
	\|w_k+v_\la\|^2 -\la\|w_k +v_\la\|_{HL}^{22_\mu^\ast} -\la \I\Om\tilde{z}_k(w_k+v_\la)^{1-\gamma}\geq -\frac{1}{n}\max\{1,\|w_k+v_\la\|\}.
\end{align}
From \eqref{e5.31} and \eqref{e5.32} we derive
\begin{align*}
	\frac{1}{2} \|w_k\|^2-\frac{1}{2_\mu^\ast}\|w_k\|^2 \leq M_1 +M_2 \|w_k\|,
\end{align*}
where $M_1,~M_2$ are positive constants. This implies that $\{w_k\}_{k\in \mb N}$ is a bounded sequence and hence $w_k \rightharpoonup w_\la$ weakly in $H_0^1(\Om)$. As in the zero altitude case we see that $w_\la$ solves $(\hat{P}_\la)$. Now we claim that $w_k \ra w_\la$ in $H_0^1(\Om)$ and that $w_\la \not\equiv 0$. Without loss of generality we assume $\mc G_\la(w_\la) \ne 0$.\\
As $\|w_k\| \leq M$, from \eqref{e5.30}, for $\aleph \in X^+$ we have $\mc G_\la^0 (w_k,\aleph-w_k) \geq \ds -\frac{M_1}{k}(1+\|\aleph\|)=o_k(1)$. Then as in zero altitude case, we get
\begin{align}\label{e5.33}
	\|w_k-w_\la\| -\la \|w_k-w_\la\|^{22_\mu^\ast} =o_k(1).
\end{align}
	 Also using Brezis-Lieb, we have
	 \begin{align}\label{e5.34}\nonumber
	 	\mc G_\la (w_k) =&\frac{1}{2}\|w_k\|^2 -\la \I\Om F(w_k)dx-\frac{\la}{2 2_\mu^\ast} \|w_k+v_\la\|_{HL}^{22_\mu^\ast}
	 	-\la \I\Om\I\Om \frac{v_\la^{2_\mu^\ast}(y)v_\la^{2_\mu^\ast-1}(x)w_k(x)}{|x-y|^\mu}dxdy\\\nonumber
	 	=& \frac{1}{2}\|w_k-w_\la\|^2+\frac12 \|w_\la\|^2 +\I\Om \na(w_k-w_\la) \na w_\la dx-\la \I\Om F(w_k)dx\\ \nonumber
	 	&-\frac{\la}{2 2_\mu^\ast} \|w_k+v_\la\|_{HL}^{22_\mu^\ast}
	 	-\la \I\Om\I\Om \frac{v_\la^{2_\mu^\ast}(y)v_\la^{2_\mu^\ast-1}(x)w_k(x)}{|x-y|^\mu}dxdy\\ \nonumber
	 	=& \frac12\|w_k -w_\la\|^2-\frac{\la}{2 2_\mu^\ast} \|w_k-v_\la\|_{HL}^{22_\mu^\ast} +\frac12 \|w_\la\|^2-\la \I\Om F(w_k)dx\\ \nonumber
	 	&-\frac{\la}{2 2_\mu^\ast} \|w_\la+v_\la\|_{HL}^{22_\mu^\ast}
	 	-\la \I\Om\I\Om \frac{v_\la^{2_\mu^\ast}(y)v_\la^{2_\mu^\ast-1}(x)w_k(x)}{|x-y|^\mu}dxdy+o_k(1)\\ \nonumber
	 	=& \frac12\|w_k -w_\la\|^2-\frac{\la}{2 2_\mu^\ast} \|w_k-v_\la\|_{HL}^{22_\mu^\ast} +\mc G_\la(w_\la)+\la \I\Om(F(w_\la)-F(w_k))dx +o_k(1)\\
	 	=& \frac12\|w_k -w_\la\|^2-\frac{\la}{2 2_\mu^\ast} \|w_k-v_\la\|_{HL}^{22_\mu^\ast} +\mc G_\la(w_\la) +o_k(1).
	 \end{align}
 Now as $\mc G_\la(w_\la) =0$, using \eqref{5.28}, \eqref{e5.30}, \eqref{e5.33} and \eqref{e5.34}, we get
 \begin{align}\label{e5.35}
 	\|w_k-w_\la\|^2 =\left(\frac{2(2n-\mu)}{n-\mu+2}\right)c_0+o_k(1)<\frac{S_{H,L}^\frac{2n-\mu}{n-\mu+2}}{\la^\frac{n-2}{n-\mu+2}}+o_k(1).
 \end{align}
Also by Hardy-Littlewood-Sobolev inequality, we have
\begin{align}\label{e5.36}
	\|w_k -w_\la\|^2\left(1-\la S_{HL}^{-2_\mu^\ast}\|w_k-w_\la\|^{22_\mu^\ast-2}\right) \leq\|w_k -w_\la\|^2-\la\|w_k -w_\la\|_{HL}^{22_\mu^\ast}=o_k(1).
\end{align}
 Thus combining \eqref{e5.35} and \eqref{e5.36}, we obtain $w_k \ra w_\la$ in $H_0^1(\Om)$. This completes the proof. \QED
\end{proof}
We are now ready to give the \\
\textbf{Proof of Theorem \ref{tmr}}: The existence of the first solution $v_\la$ for all $\la \in (0,\La^a)$ and $\gamma >0$ follows from Lemma \ref{l4.2}. The existence of second solution $w_\la$ for the same range of $\la$ and $0<\gamma<3$ follows from Lemmata \ref{l5.6} and \ref{l5.10} keeping in view Remark \ref{r5.5}. \QED
\section{Regularity Results} \label{s6}
In this section, we will discuss the regularity results. First let us recall an important inequality for nonlocal nonlinearities by Moroz and Van Schaftingen \cite{MS1}.
\begin{Lemma} \label{l6.1}
	Let $n \geq 2$, $\mu \in (0,n)$ and $\theta \in (0,n)$. If $H,~K \in L^\frac{2n}{n-\mu+2}(\mb R^n) + L^\frac{2n}{n-\mu}(\mb R^n)$, $\left(1-\frac{\mu}{n}\right)<\theta<\left( 1+\frac{\mu}{n}\right)$, then for any $\e >0$, there exists $M_{\e,\theta} \in \mb R$ such that for any $ u \in H^1(\mb R^n)$,
	\begin{align*}
	\I{\mb R^n} \left(|x|^{-\mu} \ast (H |u|^\theta)\right)K|u|^{2-\theta} dx \leq \e^2 \I{\mb R^n} |\na u|^2dx +M_{\e,\theta} \I{\mb R^n} |u|^2dx.
	\end{align*}
\end{Lemma}
We have the following Lemma which provides the $L^\infty$ estimates and boundary behaviour for the weak solutions of $(\mc P_\la)$.
\begin{Lemma}\label{l6.2}
	Let $u$ be a nonnegative weak solution of $(\mc P_\la)$. Then $u \in L^\infty(\Om)$. 
\end{Lemma}
\begin{proof}
	Let $u$ be a nonnegative weak solution of $(\mc P_\la)$. Let $\Upsilon:\mb R \ra [0,1]$ be a $C^\infty(\mb R)$ convex increasing function such that $\Upsilon^\prime(t) \leq 1$ for all $t \in [0,1]$ and $\Upsilon^\prime(t)=1$ when $t \geq 1$. Define $\Upsilon_\epsilon(t)= \e \Upsilon(\frac{t}{\e})$. Then using the fact that $\Upsilon_\e$ is smooth, we obtain $\Upsilon_\e \ra (t-1)^+$ uniformly as $\e \ra 0$. It implies
	\begin{align*}
		-\De \Upsilon_\e(u) \leq& \Upsilon_\e^\prime(u)(-\De)u \leq \chi_{\{u>1\}}(-\De)u\\
		 \leq&\chi_{\{u>1\}} \left(\la\chi_{\{u<a\}}u^{-\gamma} + \la \left(\I\Om \frac{u^{2_\mu^\ast}(y)}{|x-y|^\mu}\right) u^{2_\mu^\ast-1}\right)\\
		\leq& M\left(1 +  \left(\I\Om \frac{ u^{2_\mu^\ast}(y)}{|x-y|^\mu}\right)  u^{2_\mu^\ast-1}(x)\right).
	\end{align*}
Hence, as $\e \ra 0$, we deduce that
\begin{align}\label{6.1}
	-\De (u-1)^+ \leq M \left(1 +  \left(\I\Om \frac{ u^{2_\mu^\ast}(y)}{|x-y|^\mu}\right)  u^{2_\mu^\ast-1}(x)\right).
\end{align}
For $\tau >0$, we define $u_\tau =\min\{u,\tau\}$.
Since $w=|u_\tau|^{q-2}u_\tau \in H_0^1(\Om)$ for $q \geq 2$, we can take it as a test function in \eqref{6.1}. Now
\begin{align}\label{6.2} \nonumber
	\frac{4(q-1)}{q^2} \I\Om | \na (u_\tau)^\frac{q}{2}|^2 dx =& (q-1) \I\Om |u_\tau|^{q-2} |\na u_\tau|^2dx\\
	\leq& \I\Om \na u \na w = \I{\{u \geq 1\}} \na (u-1)^+ \na w + \I{\{0\leq u\leq 1\}} \na u\na w.
\end{align}
Note that for any $\tau >1$, 
\begin{align}\label{6.3}
	 \I{\{0\leq u\leq 1\}} \na u\na w = (q-1) \I{\{0\leq u\leq 1\}}|u_\tau|^{q-2} |\na u_\tau|^2 \leq m_1 \|u\|^2 = m_2,
\end{align}
and $m_2$ is independent of $\tau$. Taking into account \eqref{6.1} and \eqref{6.3}, we obtain from \eqref{6.2}
\begin{align*}
		\frac{4(q-1)}{q^2} \I\Om | \na (u_\tau)^\frac{q}{2}|^2 dx \leq M\I\Om \I\Om \frac{u^{2_\mu^\ast}(y)}{|x-y|^\mu}dy u^{2_\mu^\ast-1}(x) u_\tau^{q-1}(x)dx +M \I\Om |u_\tau|^{q-1} dx +m_2.
\end{align*}
If $2 \leq s <\frac{2n}{n-\mu}$, using Lemma \ref{l6.1} with $\theta =\frac{2}{q}$, there exists $M_1 >0$ such that
\begin{align*}
	\I\Om \I\Om \frac{u_\tau^{2_\mu^\ast}(y)}{|x-y|^\mu}dy u_\tau^{2_\mu^\ast-1}(x) u_\tau^{q-1}(x)dx \leq \frac{2(q-1)}{Mq^2} \I \Om | \na (u_\tau)^\frac{q}{2}|^2 dx + M_1 \I \Om | u_\tau ^\frac{q}{2}|^2 dx.
\end{align*}
Since $u_\tau \leq u$, we have
\begin{align*}
		\frac{2(q-1)}{q^2} \I\Om | \na (u_\tau)^\frac{q}{2}|^2 dx \leq& M_1 \I\Om u^q dx + M\I{A_\tau}\I\Om \frac{u^{2_\mu^\ast-1}(y) u^{q-1}(y)}{|x-y|^\mu}dy u^{2_\mu^\ast-1}(x) u^{2_\mu^\ast}(x)dx\\
		&+ M \I\Om u^{q-1}dx +m_2,
\end{align*}
where $A_\tau =\{x \in \Om : u>\tau\}$.\\
Since $2 \leq q <\ds\frac{2n}{n-\mu}$, applying the Hardy-Littlewood-Sobolev inequality again, we have
\begin{align*}
	\I{A_\tau}\I\Om \frac{u^{2_\mu^\ast-1}(y) u^{q-1}(y)}{|x-y|^\mu}dy u^{2_\mu^\ast-1}(x) u^{2_\mu^\ast}(x)dx \leq M_2 \left(\I\Om | u^{2_\mu^\ast-1} u^{q-1}|^rdx\right)^\frac{1}{r} \left(\I{A_\tau} |u^{2_\mu^\ast}|^l\right)^\frac{1}{l},
\end{align*}
where $\ds\frac{1}{r} =1 -\frac{n-\mu}{2n} -\frac{1}{q}$ and $\ds\frac{1}{l} =\frac{n-\mu}{2n} +\frac{1}{q}$. By H\"{o}lder's inequality, if $u \in L^{q} (\Om)$, then $u^{2_\mu^\ast} \in L^l(\Om)$ and $|u|^{2_\mu^\ast-1}|u|^{q-1} \in L^r(\Om)$, whence by Lebesgue's dominated convergence theorem
\begin{align*}
	\lim\limits_{\tau \ra \infty} 	\I{A_\tau}\I\Om \frac{u^{2_\mu^\ast-1}(y) u^{q-1}(y)}{|x-y|^\mu}dy u^{2_\mu^\ast-1}(x) u^{2_\mu^\ast}(x)dx =0.
\end{align*}
Finally by Sobolev embedding theory, we obtain that there exists a constant $\hat{M}$, independent of $\tau$, such that
\begin{align*}
	\left(\I \Om |u_\tau|^\frac{qn}{n-2}\right)^{1-\frac{2}{n}} \leq M_1 \I\Om u^q dx + M \I\Om u^{q-1} +\hat{M}.
	\end{align*}
Letting $\tau \ra \infty$ we conclude that $u \in L^\frac{qn}{n-2}$. By iterating over $q$ a finite number of times we cover the range $q \in [2, \frac{2n}{n-\mu})$. So we get weak solution $u \in L^q (\Om)$ for every $q\in [2, \frac{2n^2}{(n-\mu)(n-2)}]$. Thus, $u^{2_\mu^\ast}\in L^q(\Om)$ for every $q \in [\frac{2(n-2)}{2n-\mu}, \frac{2n^2}{(n-\mu)(2n-\mu)})$. Since $\frac{2(n-2)}{2n-\mu}<\frac{n}{n-\mu}<\frac{2n^2}{(n-\mu)(2n-\mu)}$, we have
\begin{align*}
	\I\Om \frac{u^{2_\mu^\ast}}{|x-y|^\mu}dy \in L^\infty(\Om)
\end{align*}
and so from \cite[Theorem 1.16]{AM}, we have $(u-1)^+ \in L^\infty(\Om)$ which imply that $u \in L^\infty(\Om)$.
 \QED
\end{proof}
 
Finally, we give the proof of\\
\textbf{Proof of Theorem \ref{trr}:} We first prove the boundary behavior. For this, we see that $u$ is a supersolution for $(\mc S_{\la,\e})$ for any $\e$. Then by applying Theorem \ref{wcp}, we get $u_{\la,\e}\leq u$ a.e. in $\Om$. Furthermore, thanks to Lemma \ref{l6.2} we see that $u$ is a subsolution to the following problem
\begin{align*}
		-\De w =\la w^{-\gamma}+ K,~w>0~\text{in}~\Om,~w=0~\text{on}~\partial \Om,
	\end{align*}
 where $K =\la K_1 |u  |_\infty^{22_\mu^\ast-1}$ and $K_1 =\left|\ds \I\Om \frac{dy}{|x-y|^\mu}\right|_\infty$ and thus $u \leq w$ a.e. in $\Om$ i.e., we have $u_{\la,\e} \leq u\leq w$ a.e. in $\Om$. Now since both $u_{\la,\e} \sim \phi_\gamma$ and $w\sim \phi_\gamma$, we have $u \sim \phi_\gamma$. Finally, the proof of H\"{o}lder continuity results follows directly from Lemma \ref{6.2} and \cite[Theorem 1.2]{GL}. \QED
		\textbf{Acknowledgement:} The first author thanks the CSIR(India) for financial support in the form of a Senior Research Fellowship, Grant Number $09/086(1406)/2019$-EMR-I. The second author is partially funded by IFCAM (Indo-French Centre for Applied Mathematics) IRL CNRS 3494.

	\end{document}